\documentclass[a4paper,11pt]{article}

\usepackage{comment}
\usepackage{amssymb,amsmath,amsthm,graphicx,mathtools,enumerate,mathrsfs}
\usepackage{fullpage,enumitem}
\usepackage{thmtools}
\usepackage{thm-restate}
\usepackage{appendix}
\usepackage{comment}
\usepackage[english]{babel}
\usepackage{caption}
\usepackage{subfigure}

\usepackage[margin=2.5cm]{geometry}
\usepackage[normalem]{ulem}
\usepackage{etoolbox}
\usepackage{scalerel}

\usepackage{tasks}

\usepackage[square,sort,comma,numbers]{natbib} 

\usepackage{makecell}
\usepackage{float}
\usepackage{caption}
\usepackage{tikz}
\usetikzlibrary{arrows}
\usetikzlibrary{patterns}
\usetikzlibrary{decorations.markings}

\usetikzlibrary{backgrounds}
\usetikzlibrary{arrows.meta}
\usetikzlibrary{shapes, arrows, calc, arrows.meta, fit, positioning}
\usetikzlibrary{%
	arrows,%
	arrows.meta,positioning,
	calc,%
	fit,%
	patterns,%
	plotmarks,%
	shapes.geometric,%
	shapes.misc,%
	shapes.symbols,%
	shapes.arrows,%
	shapes.callouts,%
	backgrounds,%
	chains,%
	topaths,%
	trees,%
	petri,%
	mindmap,%
	matrix,%
	folding,%
	fadings,%
	through,%
	positioning,%
	scopes,%
	decorations.fractals,%
	decorations.shapes,%
	decorations.text,%
	decorations.pathmorphing,%
	decorations.pathreplacing,%
	decorations.footprints,%
	decorations.markings,%
	shadows}
\usepackage{enumerate}

\usepackage{hyperref}
\hypersetup{  
	colorlinks = true,    
	linkcolor = black,      
	urlcolor = black,        
	citecolor = black,     
	linkbordercolor = black,  
}
\tikzset{middlearrow/.style={
		decoration={markings,
			mark= at position 0.5 with {\arrow[very thick]{#1}} ,
		},
		postaction={decorate}
	}
}
\tikzset{lefttlearrow/.style={
		decoration={markings,
			mark= at position 0.7 with {\arrow[very thick]{#1}} ,
		},
		postaction={decorate}
	}
}
\tikzset{leftlearrow/.style={
		decoration={markings,
			mark= at position 0.6 with {\arrow[very thick]{#1}} ,
		},
		postaction={decorate}
	}
}

\newtheorem{theorem}{Theorem}[section]
\newtheorem{conjecture}{Conjecture}[section]

\newtheorem{definition}{Definition}
\newtheorem{lemma}[theorem]{Lemma}
\newtheorem{claim}{Claim}
\newtheorem{proposition}[theorem]{Proposition}

\title{Ore-type condition for antidirected Hamilton cycles in oriented graphs}

\author{Junqing Cai\thanks{JQC: School of Mathematical Science, Tianjin Normal University, Tianjin 300387, China. \texttt{caijq09@163.com}. The author was supported by National Natural Science Foundation of China (No.12371356).} \and    
	Guanghui Wang\thanks{GHW: School of Mathematics,  Shandong University, Jinan 250100, China.   \texttt{ghwang@sdu.edu.cn}. The author was supported by National Natural Science Foundation of China (No.12231018) and State Key Laboratory of Cryptography and Digital Economy Security.} \and
	Yun Wang\thanks{YW: Data Science Institute, Shandong University, Jinan 250100, China. \texttt{yunwang@sdu.edu.cn}. The author was supported by China Postdoctoral Science Foundation (No. 2025M773101) and National Natural Science Foundation of China (No.12501489).}
	\and
	Zhiwei Zhang\thanks{ZWZ: Interdisciplinary Center, Shandong University, Jinan 250100, China. \texttt{zhiweizh@mail.sdu.edu.cn}.}}
\date{}

\begin{document} 
	\maketitle
	\begin{abstract}
		An antidirected cycle in a digraph $G$ is a subdigraph whose underlying graph is a cycle, and in which no two consecutive edges form a directed path in $G$. Let $\sigma_{+-}(G)$ be the minimum value of $d^+(x)+d^-(y)$ over all pairs of vertices $x, y$ such that there is no edge from $x$ to $y$, that is, $$\sigma_{+-}(G)=\min\{d^+(x)+d^-(y): \{x,y\}\subseteq V(G), xy\notin E(G)\}.$$
		In 1972, Woodall extended Ore's theorem to digraphs by showing that every digraph $G$ on $n$ vertices with $\sigma_{+-}(G)\geqslant n$ contains a directed Hamilton cycle. Very recently, this result was generalized to oriented graphs under the condition $\sigma_{+-}(G)\geqslant(3n-3)/4$. In this paper, we give the exact Ore-type degree threshold for the existence of antidirected Hamilton cycles in oriented graphs.  More precisely, we prove that for sufficiently large even integer $n$, every oriented graph $G$ on $n$ vertices with $\sigma_{+-}(G)\geqslant(3n+2)/4$  contains an antidirected Hamilton cycle. Moreover, we show that this degree condition is best possible.
	\end{abstract}
	
	\maketitle
	
	\noindent{\bf Keywords:}  \texttt{antidirected cycle,  Ore-type condition, oriented graph, Hamilton cycle}

	\section{Introduction}\label{SEC-introduction}
	
	The exploration of Hamiltonicity of graphs has a long-standing history and well-studied in graph theory, starting with graphs, where numerous classical results have laid the foundation for subsequent research. As the study progressed, extending these results to digraphs and oriented graphs became a central pursuit, driven by both theoretical interest and practical applications. In this introduction, we first review key developments in the study of Hamilton cycles in graphs, then turn to the more intricate setting of digraphs and oriented graphs (digraphs without directed 2‑cycles), providing context for the contributions of this paper.

	\subsection{Hamilton cycles in graphs}
	
	A Hamilton cycle in a  graph $G$ is a cycle that visits every vertex of $G$ exactly once. A graph containing such a cycle is called Hamiltonian. Determining whether a graph is Hamiltonian is an NP-complete problem, which motivates the search for sufficient conditions that guarantee Hamiltonicity, particularly those based on degree conditions.
	
	A cornerstone result is Dirac’s theorem \cite{diracPLMS2}, which states that every graph on $n\geqslant 3$ vertices with minimum degree at least $n/2$ is Hamiltonian. Ore \cite{oreAMM67} later strengthened Dirac’s theorem by showing that a graph on $n\geqslant 3$ vertices is Hamiltonian if $d(x)+d(y)\geqslant n$ for every pair of non‑adjacent vertices $x$ and $y$. Further generalizations include Fan’s theorem \cite{fan1984new}, which considers all pair of vertices with distance two; Pósa’s theorem \cite{posaPMIHAS7}, which allows many vertices with small degrees; and Chvátal’s theorem \cite{chvatalJCTB12}, which gives a complete characterization of degree sequences that force Hamiltonicity. For comprehensive surveys, see \cite{chibaGC34} and \cite{gouldGC19}.
	
	\subsection{Hamilton cycles in digraphs and oriented graphs}
	
	Hamiltonicity in digraphs, presents distinct challenges and has attracted considerable attention. A natural analogue of the minimum degree is the \emph{minimum semidegree} $\delta^0(G)$, defined as the minimum over all in‑degrees and out‑degrees of the vertices of $G$. Ghouila-Houri \cite{ghouilaCRASP25} extended Dirac’s theorem to digraphs by proving that every digraph on $n$ vertices with $\delta^0(G)\geqslant n/2$ contains a directed Hamilton cycle. For oriented graphs, Keevash, K{\"u}hn and Osthus \cite{keevashJLMS79} established the tight bound $\delta^0(G)\geqslant (3n-4)/8$ for the existence of a directed Hamilton cycle in large oriented graphs.
	
	In 1972,  Woodall \cite{woodallPLMS24} obtained a directed analogue of Ore’s theorem. For a digraph $G$, define the Ore-type parameter
	$$\sigma_{+-}(G)=\min\{d^+(x)+d^-(y): \{x,y\}\subseteq  V(G), xy\notin E(G)\}.$$
	
	\begin{theorem}[\cite{woodallPLMS24}]\label{THM-woodall}
		Every $n$-vertex digraph $G$ with $\sigma_{+-}(G)\geqslant n$ contains a directed Hamilton cycle.
	\end{theorem}
	
	For oriented graphs, Kelly, K{\"u}hn and Osthus \cite{kellyCPC17} proved that for every $\varepsilon >0$, there exists an integer $n_0$ such that every oriented graph on $n\geqslant n_0$ vertices with $\sigma_{+-}(G)\geqslant(3/4+\varepsilon )n$ contains a directed Hamilton cycle. Very recently, Chang et al.  \cite{changarXiv2025} gave the exact Ore-type degree threshold for a directed Hamilton cycle in oriented graphs.
	
	\begin{theorem}[\cite{changarXiv2025}]\label{THM-changarXiv2025}
		There exists an integer $n_0$ such that every oriented graph $G$ on $n\geqslant n_0$ vertices with $\sigma_{+-}(G)\geqslant(3n-3)/4$ contains a directed Hamilton cycle.
	\end{theorem}
	
	Instead of insisting on a consistently directed Hamilton cycle, one may ask whether a digraph contains some orientation of a Hamilton cycle under suitable degree conditions. H{\"a}ggkvist and Thomason \cite{haggkvist1995oriented} proved that every sufficiently large $n$-vertex digraph $G$ with $\delta^0(G) > n/2 + n^{5/6}$ contains every orientation of a Hamilton cycle. Twenty years later, Debiasio et al. \cite{debiasioSJDM29} completely determined the semidegree threshold for this property in digraphs, showing that every large digraph $G$ with $\delta^0(G)\geqslant  n/2$ contains every orientation of a Hamilton cycle, with the possible exception of the antidirected orientation. For oriented graphs, Wang, Wang and Zhang \cite{wangarxiv2025} proved that the minimum semidegree threshold for all possible orientations of a Hamilton cycle is $\delta^0(G)\geqslant (3n-1)/8$. Very recently, Debiasio and Treglown \cite{debiasioarXiv2025} considered the minimum degree $\delta(G)$ of a digraph, i.e., the minimum number of edges incident to a vertex, and proved that for all \(\eta > 0\), every sufficiently large digraph on $n$ vertices with $\delta(G) > (1 + \eta)n$ contains every orientation of a Hamilton cycle, with a specified exception concerning directed cycles in non-strongly connected digraphs.
	
	Among all orientations, \emph{antidirected} Hamilton cycles, i.e., cycles in which no two consecutive edges have the same direction, are of particular interest. Their existence depends delicately on both local and global structures, often requiring more refined degree conditions. Diwan et al. \cite{diwanDM311} showed that determining whether a digraph contains an antidirected cycle-factor is NP-complete. In 2015, DeBiasio and Molla \cite{debiasioEJC22} proved that every large digraph $G$ of even order $n$ with $\delta^0(G)\geqslant n/2$ contains an antidirected Hamilton cycle except for two explicit counterexamples. For oriented graphs, a consequence of the result in \cite{wangarxiv2025} is that every large oriented graph $G$ on even $n$ vertices with $\delta^0(G)\geqslant (3n-1)/8$  contains an antidirected Hamilton cycle. Furthermore, Stein and Z\'arate-Guer\'en \cite{stein2024antidirected} studied antidirected cycles of intermediate lengths, proving that for any \(\eta \in (0,1)\), there exists \(n_0\) such that for all \(n > n_0\) and \(k > \eta n\), every oriented graph \(G\) on \(n\) vertices with \(\delta^0(G) > (1+\eta)k\) contains any antidirected cycle of length at most \(k\).
	
	\subsection{Main result}
	
	In this paper, we establish an Ore‑type degree condition that guarantees an antidirected Hamilton cycle in oriented graphs. Our main result is as follows. 
	
	\begin{theorem}\label{THM-mainthm}
		There exists an integer $n_0$ such that every oriented graph $G$ on even $n\geqslant n_0$ vertices with $\sigma_{+-}(G)\geqslant(3n+2)/4$  contains an antidirected Hamilton cycle.
	\end{theorem}
	
	The degree bound in Theorem \ref{THM-mainthm} is best possible, as shown by the following construction.
	
	\begin{proposition}\label{PROP-degreesharp}
		For any even integer $n \geqslant4$, there are infinitely many oriented graphs $G$ on $n$ vertices with $\sigma_{+-}(G)= \lceil (3n+2)/4 \rceil - 1$ that do not contain an antidirected Hamilton cycle.
	\end{proposition}
	
	\paragraph*{Organization of the paper.} Section \ref{SEC: notation-sketch} introduces supplementary notation. Section \ref{SEC-pre} presents key technical tools and preliminary results. Section \ref{SEC: non-expander case} provides the proof of Theorem \ref{THM-mainthm}.  Section \ref{SEC: pf-of-PROP-degreesharp} discusses the sharpness of the degree condition through the proof of Proposition \ref{PROP-degreesharp}. Finally, Section \ref{SEC-remark} concludes this paper and proposes several open problems for future research.

	\section{Notation}\label{SEC: notation-sketch}
	
	In this section, we provide some basic notation and terminology. Notations not defined here, we refer to \cite{bang2009}. A {\bf digraph}  is not allowed to have parallel edges or loops and an {\bf oriented graph} is a digraph with no cycle of length 2. For two vertices $x,y$ in a digraph, we write $xy$ to denote the edge from $x$ to $y$. A path or cycle is called {\bf directed} if all of its arcs are oriented in the same direction, and {\bf antidirected} if it contains no directed path of length 2. Note that an antidirected cycle must have an even number of vertices. 
	
	For a digraph $G=(V,E)$, we use $V(G)$ and $E(G)$ to denote the set of vertices and the set of arcs of $G$, respectively. The {\bf order} of $G$, i.e., the number of vertices, is written as $|G|$, and the {\bf size} of $G$, i.e., the number of edges, is denoted by $e(G)$. The subdigraph induced by a vertex subset $X$ is written as $G[X]$, and we define $G-X = G[V\backslash X]$. For two subsets $A,B\subseteq V(G)$, let $E(A,B)$ denote the set of edges from $A$ to $B$ in $G$ and set $e(A,B) = |E(A,B)|$. If $B=A$, then we write $E(A)$ for $E(A,A)$ and $e(A)$ for $|E(A)|$.

	For a subdigraph $H\subseteq G$ and a vertex $v\in V(G)$, the {\bf out-neighborhood} of $v$ in $H$ is $N_G^+(v,H)=\{u\in V(H): vu\in E(G)\}$, and its {\bf out-degree} is $d_G^+(v,H) = |N_G^+(v,H)|$. The {\bf in-neighborhood} $N_G^-(v,H)$ and  {\bf in-degree} $d_G^-(v,H)$ are defined analogously. We write $d(v,H)=d^+(v,H)+d^-(v,H)$ for the degree of $v$ in $H$. The notation $d^{\pm}(v,H)\geqslant d$ means that both $d^+(v,H)$ and $d^-(v,H)$ are at least $d$. When the digraph $G$ is clear from the context, we usually omit the subscript and simply write $N^+(v,H), d^+(v,H)$, etc. When $H=G$, we abbreviate further to $N^+(v), d^+(v)$, and so on.

	Let $L$ be a path or cycle of an oriented graph $G$,  the {\bf length} of $L$ is the number of edges it contains. For a vertex $v\in V(L)$, it is a {\bf sink} vertex of $L$ if $d^+(v,L)=0$ and a {\bf source} vertex if $d^-(v,L)=0$.  Given subsets $X_1,X_2,\ldots,X_l\subseteq V(G)$, we say that $R= v_1v_2\cdots v_l$ has {\bf form} $X_1X_2\cdots X_l$ if $v_i\in X_i$ for all $i\in[l]$. For simplicity, we will say that $R$ has form $X^{l}$ if every vertex of $R$ belongs to $X$.

	Let $f,g: \mathbb{N}\longrightarrow \mathbb{R}^+$ be two functions. We write $f(n)=O(g(n))$ if there exists a constant $K>0$ such that $f(n) \leqslant Kg(n)$ for all sufficiently large $n$. The notation $\alpha \leqslant  \beta \leqslant  \gamma$ means that there exist increasing functions $f$ and $g$ such that $\beta \leqslant  f(\gamma)$ and $\alpha \leqslant  g(\beta)$.  For a positive integer $n$, $[n]$ denotes the set $\{1, 2, \ldots, n\}$.

	
	\section{Preparation}\label{SEC-pre}

	This section collects the definitions, tools, and preliminary results required for our proofs. We
	begin with a probabilistic partitioning lemma, followed by classical graph-theoretic facts, and
	then introduce the central concept of robust expansion along with its key properties. 
	The first lemma is a standard consequence of the Chernoff bound. It will be used to partition
	vertex sets while approximately preserving degree information.
	\begin{lemma}[\cite{debiasioEJC22}]\label{LEM-randompartition}
		For any $\varepsilon > 0$, there exists an integer $n_0$ such that the following holds. Suppose that $G$ is a digraph on $n \geqslant  n_0$ vertices. Let $S\subseteq V(G)$ and $|S|\geqslant m$, and define $K= m/|S|$. Then there exists a subset $T \subseteq S$ of order $m$ such that for every $v\in V(G)$,
		\begin{align*}|d^{\pm}(v,T) - Kd^{\pm}(v,S)| \leqslant   \varepsilon n \mbox{ and }|d^{\pm}(v,S\backslash T)- (1-K)d^{\pm}(v,S)| \leqslant   \varepsilon n.\end{align*}
	\end{lemma}

	The following elementary proposition allows us to find a subgraph with high minimum degree. It is frequently used to pass from an average degree condition to a more manageable minimum degree condition.
	
	\begin{proposition}\label{PRO-findsubgraph}
		Every graph $G$ contains a subgraph $H$ with $\delta(H) \geqslant e(G)/|G|$.
	\end{proposition}

	The concept of a robust expander is fundamental to many recent results on Hamiltonicity. We recall its definition and basic properties, as it serves as the main structural condition in our absorbing arguments.	
	
	\begin{definition} \emph{(Robust outexpander)}
		Given $0<\nu\leqslant   \tau <1$. For an $n$-vertex  digraph $G$ and a subset $S\subseteq V(G)$, the $\nu$-out-neighborhood of $S$ in $G$, denoted by $RN_{\nu,G}^+(S)$, is the set of vertices that have at least $\nu n$ in-neighbors in $S$.  We say that $G$ is a robust $(\nu,\tau)$-outexpander if $$|RN_{\nu,G}^+(S)|\geqslant  |S|+\nu n$$ for every $S\subseteq V(G)$ with $\tau n< |S|< (1-\tau)n$.
	\end{definition}

	The next lemma, due to Kühn and Osthus \cite{kuhnAM237}, provides a crucial link between a semidegree condition and robust expander. It ensures that sufficiently dense oriented graphs automatically possess this expansive property.	

	\begin{lemma}[\cite{kuhnAM237}]\label{LEM-semitoexpander}
		Let $0<1/n\ll\nu\ll\tau\leqslant \varepsilon/2\leqslant  1$. Every $n$-vertex oriented graph $G$ with $\delta^0(G)\geqslant  (3/8+\varepsilon)n$ is a robust $(\nu,\tau)$-outexpander.
	\end{lemma}

	The power of robust expansion is captured in the following theorem by Taylor \cite{taylor2013}, which states that robust expanders with linear minimum semidegree are useful for all orientations of Hamilton cycles.
	
	\begin{theorem}[\cite{taylor2013}]\label{THM-expanderanyori}
		Let $1/n\ll \nu\leqslant   \tau\ll \gamma <1$. Suppose  $G$ is an $n$-vertex digraph with $\delta^0(G)\geqslant  \gamma n$. If $G$ is a robust $(\nu,\tau)$-outexpander, then $G$ contains every possible orientation of a Hamilton cycle.
	\end{theorem}

	A slight adaptation of the proof of Theorem \ref{THM-expanderanyori} yields a more flexible version for finding Hamilton paths with prescribed endpoints, a tool essential for connecting absorption structures. Its proof can be found in \cite{debiasioarXiv2025}.
	
	\begin{theorem}[\cite{debiasioarXiv2025}]\label{THM-HPfixendv}
		Let $1/n \ll \nu\leqslant   \tau\ll \gamma <1$. Suppose $G$ is an $n$-vertex  digraph  with $\delta^0(G)\geqslant  \gamma n$. If $G$ is a robust $(\nu,\tau)$-outexpander, then for any two distinct vertices $x,y\in V(G)$ and any oriented path $P$ on $n$ vertices, there is a copy of $P$ in $G$ that starts at $x$ and ends at $y$.
	\end{theorem}
	
	Our main theorem assumes an Ore-type condition. The following simple proposition shows that this global condition implies a linear minimum semidegree, allowing us to connect it to the previous lemmas.

	\begin{proposition}\label{PROP:Ore-implies-semideg}
		Let $G$ be an $n$-vertex  oriented graph. If $\sigma_{+-}(G)\geqslant (n+3\gamma n)/2$, then $G$ has minimum semidegree $\delta^0(G)\geqslant \gamma n$. 
	\end{proposition}
	\begin{proof}
		By symmetry, it suffices to show that every vertex of $G$ has in-degree at least $\gamma n$ in $G$. 
		Suppose to the contrary that there exists a vertex $v\in V(G)$ with $d^-(v) <\gamma n$. Set $X=N^-(v)$ and $Y=V(G)\backslash (X\cup \{v\})$. For any vertex $y\in Y$, the Ore-type condition implies that $d^+(y) \geqslant (n +3\gamma n)/2 -d^-(v)\geqslant (n +3\gamma n)/2 -\gamma n=(n+\gamma n)/2$. Thus $e(Y,G)>|Y|(n+\gamma n)/2$.      
		On the other hand,  $e(Y,G)\leqslant |X||Y| + |Y|(|Y|-1)/2=|Y|(2|X|+|Y|-1)/2<|Y|(n+\gamma n-1)/2$, a contradiction.                
	\end{proof}
	
	Combining Proposition \ref{PROP:Ore-implies-semideg} with Lemma \ref{LEM-semitoexpander} and Theorem \ref{THM-expanderanyori} immediately gives the following corollary, which handles the ``expander case" of our main proof.

	\begin{lemma}\label{LEM:expander-case}
		Let  $0<1/n\ll \nu\leqslant   \tau\ll \gamma <1$ and let $G$ be an $n$-vertex  oriented graph. If $\sigma_{+-}(G)\geqslant (n+3\gamma n)/2$  and $G$ is a robust $(\nu,\tau)$-outexpander, then $G$ contains every possible orientation of a Hamilton cycle. In particular, $G$ has an antidirected Hamilton cycle when $n$ is even.    
	\end{lemma}
	
	Finally, we state a useful discrepancy lemma. It translates bounds on the total number of edges between two sets into information about the individual degrees of most vertices, which is crucial for analyzing extremal cases.
	
	\begin{lemma}[\cite{wangarxiv2025}]\label{LEM:edges-to-degree} 
		Let $\varepsilon$ be a real with $\varepsilon \ll 1$. Suppose that  $G$ is a digraph and  $X,Y$ are two disjoint subsets of $V(G)$. 
		
		\emph{(i)} If $e(X,Y)\geqslant|X||Y|-O(\varepsilon n^2)$, then there are at most $\varepsilon^{1/3} n$ vertices $x$ in $X$ with $d^+(x,Y)\leqslant |Y|-\sqrt{\varepsilon}n$  and at most $\varepsilon^{1/3} n$  vertices $y$ in  $Y$ with $d^-(y,X)\leqslant |X|-\sqrt{\varepsilon}n$.
		
		\emph{(ii)} If $e(X,Y)\leqslant O(\varepsilon n^2)$, then there are at most $\varepsilon^{1/3} n$ vertices $x$ in $X$ with $d^+(x,Y)\geqslant\sqrt{\varepsilon}n$ and at most $\varepsilon^{1/3} n$  vertices $y$ in  $Y$ with $d^-(y,X)\geqslant\sqrt{\varepsilon}n$.
	\end{lemma}

	\section{Proof of Theorem \ref{THM-mainthm}}\label{SEC: non-expander case}
	
	Note that the proof of Theorem \ref{THM-mainthm} can be split into two cases, depending on whether the given oriented graph $G$ is a robust $(\nu,\tau)$-outexpander or not. If $G$ is a robust $(\nu,\tau)$-outexpander, then the result follows immediately from Lemma \ref{LEM:expander-case}. It remains to consider the case when $G$ is not a robust $(\nu,\tau)$-outexpander. This case is covered by the following theorem, whose proof will be given in the end of the section.
	
	\begin{theorem}\label{THM:non-expander-case}
		Let $0<1/n\ll\nu \leqslant   \tau \ll 1$. Suppose $G$ is an $n$-vertex oriented graph with $\sigma_{+-}(G)\geqslant (3n+2)/4$. If $G$ is not a robust $(\nu,\tau)$-outexpander, then it contains an antidirected Hamilton cycle.
	\end{theorem}

	As mentioned before, to complete the proof of Theorem \ref{THM-mainthm}, it suffices to analyze the non-expander case. In this case, we first show that $G$ is structurally ``close'' to a well-behaved configuration. To formalize the notion of ``closeness'' to such a configuration, we introduce the following definition. Let $A,B,C,D$ be disjoint subsets of $V(G)$ with $V(G)=A\cup B\cup C\cup D$. For simplicity, we refer to $(A,B,C,D)$ as a partition of $G$.
	
	\begin{definition} (Nice partition)
		\label{DEF-nicep}
		Let $\varepsilon$ be a positive constant and let $G$ be an oriented graph on $n$ vertices. A partition $(A,B,C,D)$ of $V(G)$ is called an $\varepsilon$-\emph{nice partition} if the sets $A,B,C,D$ satisfy:
		\begin{enumerate}[label =\upshape \textbf{(NP\arabic{enumi})}, ref=\upshape (NP\arabic{enumi})]
			\setlength{\itemindent}{1.5em}
			\item $|A|\leqslant |C|$ and $|A| + |C| = n/2 \pm O(\varepsilon n)$;\label{NP1}
			\item $|B|, |D| = n/4 \pm O(\varepsilon n)$;\label{NP2}
			\item $e(A\cup D,C\cup D)\leqslant \varepsilon^2n^2$. \label{NP3}
		\end{enumerate}
	\end{definition}
	
	The following lemma shows that any sufficiently large oriented graph $G$ with high $\sigma_{+-}(G)$ value that fails to be a robust outexpander must admit a nice partition.
	
	\begin{lemma}\label{LEM:non-expander-case}
		Let $0<1/n\ll\nu\ll \tau,\varepsilon  \ll 1$. Suppose $G$ is an $n$-vertex oriented graph with $\sigma_{+-}(G)\geqslant  3n/4-O(\varepsilon n)$. If $G$ is not a robust $(\nu,\tau)$-outexpander, then $V(G)$ has an $\varepsilon$-nice partition $(A,B,C,D)$.
	\end{lemma}
	
	\begin{proof}
		Since $G$ is not a robust $(\nu,\tau)$-outexpander, there exists a subset $S\subseteq V(G)$ with $\tau n< |S|< (1-\tau)n$ such that $|RN_{\nu,G}^+(S)|<|S|+\nu n$. Set $A=RN_{\nu,G}^+(S)\cap S$, $B=RN_{\nu,G}^+(S)\backslash  S$, $C=V(G)\backslash  (RN_{\nu,G}^+(S)\cup S)$ and $D=S\backslash  RN_{\nu,G}^+(S)$. Then $(A,B,C,D)$ is a partition of $V(G)$ and 
		\begin{align}\label{EQ-BD}
			|B|<|D|+\nu n.
		\end{align}
		
		Without loss of generality, we may assume $|A|\le |C|$; otherwise, reverse all edges of $G$ and exchange the labels of $A$ and $C$ if necessary.
		
		By the definition of $RN_{\nu,G}^+(S)$, every vertex outside $RN_{\nu,G}^+(S)$ has fewer than $\nu n$ in-neighbors in $S$. Consequently, $e(A\cup D,C\cup D)\leqslant \nu n^2\leqslant \varepsilon^2n^2$ since $\nu\ll\varepsilon$. By Lemma \ref{LEM:edges-to-degree} (ii), there exist two subsets $A^L\subseteq A$, $C^L\subseteq C$ with  $|A^L|,|C^L|\leqslant \nu^{1/3}n$ such that $d^+(a,C\cup D)\leqslant \sqrt{\nu} n$ for each $a \in A\backslash A^L$ and  $d^-(c,A\cup D)\leqslant \sqrt{\nu} n$ for each  $c \in C\backslash C^L$. Furthermore, there exists a subset $D^L\subseteq D$ with $|D^L|\leqslant 2\nu ^{1/3}n$ such that $d^+(d,C\cup D)\leqslant \sqrt{\nu}n$ and $d^-(d,A\cup D)\leqslant \sqrt{\nu}n$ for each vertex $d\in D\backslash D^L$.

		To complete the proof, it remains to prove \ref{NP1} and \ref{NP2}. We begin by estimating the size of $D$.

		\begin{claim}\label{CLM-D}
			$\tau n/2<|D|\leqslant n/4+O(\varepsilon n)$.
		\end{claim}

		\begin{proof}
			We first prove the lower bound $|D| >\tau n/2$. Suppose to the contrary that $|D| \leqslant \tau n/2$. By inequality (\ref{EQ-BD}) and the fact that  $\nu \ll \tau$, we have $|B|\leqslant 2\tau n/3$. Moreover, it follows by $\tau n< |S|< (1-\tau)n$ that  $|A|= |S|-|D|\geqslant \tau n/2$ and $|C|=n-|S|-|B| \geqslant \tau n/3$. 
			
			Let $a\in A\backslash A^L$ be a vertex with minimum out-degree in $G[A\backslash A^L]$. Clearly, $d^+(a,A\backslash A^L)\leqslant |A\backslash A^L|/2$ as $G$ is an oriented graph. Then 
			$d^+(a)=d^+(a,A\backslash A^L)+d^+(a,A^L)+d^+(a,B)+d^+(a,C\cup D)\leqslant |A\backslash A^L|/2 +|A^L|+|B|+\sqrt{\nu}n\leqslant|A|/2+O(\tau n)$. Similarly, there is a vertex $c\in C\backslash(C^L\cup N^+(a,C))$ such that $d^-(c)\leqslant |C|/2+O(\tau n)$. Thus, $d^+(a)+d^-(c)\leqslant n/2+O(\tau n)$, which contradicts the Ore-type degree condition as $ac$ is not an edge of $G$ (since $c\in C\backslash(C^L\cup N^+(a,C))$). Therefore, $|D| >\tau n/2$.

			Next, we prove the upper bound $|D|\leqslant n/4+O(\varepsilon n)$. Recall that $D>\tau n/2$ and $|D^L|\leqslant 2\nu ^{1/3}n$. Then $|D\backslash D^L|\geqslant \tau n/3$. 
			Since $e(D)\leqslant e(A\cup D,C\cup D)\leqslant \nu n^2$, the induced subgraph $G[D\backslash D^L]$ is not a tournament. Hence, there are two vertices $d_1,d_2\in D\backslash D^L$ such that $d_1d_2\notin E(G)$ and $d_2d_1 \notin E(G)$. By the Ore-type degree condition, we get 
			$d(d_1)+d(d_2)\geqslant 3n/2-O(\varepsilon n)$.
			$d(d_1)+d(d_2)\leqslant 2(|A|+|B|+|C|+2\sqrt{\nu}n)=2(n-|D|+2\sqrt{\nu}n)$. Then the upper and lower bounds imply that $|D|\leqslant n/4 + O(\varepsilon n)$. This completes the proof of the claim.
		\end{proof}
		
		\begin{claim}\label{CLM-BD}
			$|B|= n/4\pm O(\varepsilon n)$.
		\end{claim}
		
		\begin{proof} By Claim \ref{CLM-D} and the inequality (\ref{EQ-BD}), we have $|B|\leqslant n/4+ O(\varepsilon n)$ and thus $|A|+|C|\geqslant n/2- O(\varepsilon n)$. Next we divide the proof into the following two cases based on the value of $|A|$.
			
			\medskip
			\noindent\textbf{Case 1: $|A| \leqslant \tau n/2$.}
			\medskip
			
			Choose a vertex $d \in D \backslash  D^L$ such that $d^+(d, C \cup D) \leqslant \sqrt{\nu}n$. Since $|C| \geqslant n/2 - O(\varepsilon n) - |A| \geqslant n/3$, we can choose a vertex $c \in C \backslash  (C^L \cup N^+(d,C))$ with minimum in-degree in $G[C \backslash  (C^L \cup N^+(d,C))]$. So $d^-(c, C) \leqslant |C|/2 + |C^L \cup N^+(d,C)|$. By the Ore-type degree condition,  we have 
			\begin{align*}
				3n/4 - O(\varepsilon n)&\leqslant d^+(d) + d^-(c)\\
				&\leqslant \left(|A| + |B| + \sqrt{\nu}n\right) + (|B| +|C|/2 + O(\nu^{1/3}n) +  \sqrt{\nu}n) \\
				&\leqslant (|A|+|B|+|C|+|D|)/2+ |A|/2+|B|+O(\nu^{1/3}n) \\
				&\leqslant n/2 + |B| + O(\varepsilon n).
			\end{align*}
			This implies that $|B| \geqslant n/4 - O(\varepsilon n)$ and thus  $|B| =n/4 \pm O(\varepsilon n)$.

			\medskip
			\noindent\textbf{Case 2: $|A| > \tau n/2$.}
			\medskip
			
			Choose a vertex $a \in A \backslash  A^L$ with minimum out-degree. Then $d^+(a, A) \leqslant |A|/2 + |A^L|$, and $d^+(a, C \cup D) \leqslant \sqrt{\nu}n$. Since $|C|\geqslant |A|> \tau n/2$, there exists a vertex $c \in C \backslash  (C^L\cup N^+(a,C))$ with minimum in-degree and thus $d^-(c, C) \leqslant |C|/2 + |C^L| + \sqrt{\nu}n$. By the Ore-type degree condition,  we have 
			\begin{align*}
				3n/4 - O(\varepsilon n)&\leqslant d^+(a) + d^-(c) \\
				&\leqslant (|A|/2+ |B| +  \sqrt{\nu}n) + (|B| + |C|/2 + O(\nu^{1/3}n)) \\
				&\leqslant (|A|+|B|+|C|+|D|)/2 +|B|+O(\nu^{1/3}n)\\
				&\leqslant n/2  + |B|+O(\varepsilon n).
			\end{align*}
			This implies that $|B|\geqslant n/4-O(\varepsilon n)$ and thus  $|B|= n/4\pm O(\varepsilon n)$.
		\end{proof}           
		
		By the above two claims and inequality (\ref{EQ-BD}),  we have  $n/4-O(\varepsilon n) \leqslant |B|-\nu n <|D|\leqslant n/4 +O(\varepsilon n)$. Moreover, $|A| + |C| = n - |B| - |D|= n/2 \pm O(\varepsilon n)$, which completes the proof.
	\end{proof}
	
	Having established the existence of a nice partition, we now introduce a classification of vertices based on their degree properties relative to this partition.
	
	\begin{definition} \emph{(Good/bad vertex)}\label{DEF-goodvtx}
		For a partition $(A,B,C,D)$ of an oriented graph $G$, a vertex $v$ is called a $\delta$-\textit{good} vertex of $G$ if it satisfies the following conditions, according to its part:		
		\begin{enumerate}[label =\upshape \textbf{(G\Alph*)},ref =\upshape (G\Alph*)]
			\setlength{\itemindent}{1.5em}
			\item if $v\in A$: $d^\pm(v,A)\geqslant |A|/2 - \delta n$, $d^+(v,B)\geqslant |B| -\delta n$ and $d^-(v,D) \geqslant |D| -\delta n$;\label{GA}
			\item if $v\in B$: $d^+(v,C)\geqslant |C| -\delta n $ and $ d^-(v,A) \geqslant |A| -\delta n$;\label{GB}
			\item if $v\in C$: $d^\pm(v,C)\geqslant |C|/2 - \delta n$, $d^+(v,D)\geqslant |D| -\delta n$ and $d^-(v,B) \geqslant |B| -\delta n$;\label{GC}
			\item if $v\in D$: $d^+(v,B)\geqslant |C|/2 - \delta n$, $d^-(v,B) \geqslant |A|/2 -\delta n$, $d^+(v,A)\geqslant |A| -\delta n $ and $ d^-(v,C) \geqslant |C| -\delta n$.\label{GD}
		\end{enumerate}
		Moreover, a vertex is called $\delta$-\textit{bad} if it is not $\delta$-good.
	\end{definition}
	
	The next lemma shows that for  an oriented graph $G$ with  $\sigma_{+-}(G) \geqslant 3n/4-O(\varepsilon n)$, the number of bad vertices in any nice partition of $G$ is necessarily small. Its proof is similar to the proof of Lemma \ref{LEM:non-expander-case}.
	
	\begin{lemma}\label{LEM-findbad}
		Given $\varepsilon>0$, there exists a constant $\delta \gg \varepsilon$ such that if an oriented graph $G$ satisfying $\sigma_{+-}(G) \geqslant 3n/4-O(\varepsilon n)$, then every $\varepsilon$-nice partition of $G$ contains at most $\delta n$ $\delta$-bad vertices.
	\end{lemma}
	
	\begin{proof}
		Let $(A,B,C,D)$ be an $\varepsilon$-nice partition of $V(G)$.  Applying Lemma \ref{LEM:edges-to-degree} (ii) to $X=A\cup D$ and $Y=C\cup D$, there are three sets $A^L\subseteq A, D^L\subseteq D$ and $C^L\subseteq C$ satisfying
		
		\begin{enumerate}[label =\upshape \textbf{(Y\arabic{enumi})},ref=\upshape (Y\arabic{enumi})]
			\setlength{\itemindent}{1.5em}
			\item $|A^L|,|D^L|,|C^L|\leqslant 2 \varepsilon^{2/3} n$; \label{Y1}
			\item $d^+(u,C\cup D) \leqslant \varepsilon n$ for each $u\in (A\backslash A^L)\cup (D\backslash D^L)$; \label{Y2}
			\item $d^-(v,A\cup D) \leqslant  \varepsilon  n$ for each $v\in (C\backslash C^L)\cup (D\backslash D^L)$. \label{Y3}
		\end{enumerate}

		Next, we finish the proof by proving the following two claims and some additional arguments.
		\begin{claim}
			There are at most $\delta n/4$ $\delta$-bad vertices in $D$.
		\end{claim}
		\begin{proof}
			Let $d^{\ast}$ be an arbitrary vertex in $D\backslash D^L$. By \ref{Y2} and \ref{Y3}, the vertex $d^{\ast}$ has at most $2\varepsilon n$ neighbors in $D$. 
			Then there is a vertex $d^{\ast\ast}\in D\backslash (D^L\cup\{d^{\ast}\})$ that is not adjacent to $d^{\ast}$. 
			By \ref{Y1}, there are at least $|D| - 3\varepsilon^{2/3}n$  choices for $d^{\ast\ast}$. 
			Since $d^{\ast}d^{\ast\ast}\notin E(G)$ and $d^{\ast\ast}d^{\ast}\notin E(G)$, the Ore-type condition implies that 
			$d(d^{\ast})+d(d^{\ast\ast})\geqslant 3n/2-O(\varepsilon n).$
			On the other hand, since $d^{\ast}, d^{\ast\ast} \notin D^L$, \ref{Y2}-\ref{Y3} yield
			$d(d^{\ast})+d(d^{\ast\ast})\leqslant 4\varepsilon n + d^-(d^{\ast},C)+ d^-(d^{\ast\ast},C)+2|B|+d^+(d^{\ast},A)+d^+(d^{\ast\ast},A).$
			Recall that $|B|=n/4\pm O(\varepsilon n)$ and $|A|+|C|=n/2\pm O(\varepsilon n)$.
			Comparing the upper and lower bounds of $d(d^{\ast})+d(d^{\ast\ast})$, we obtain
			$d^-(d^{\ast},C)+ d^-(d^{\ast\ast},C)+d^+(d^{\ast},A)+d^+(d^{\ast\ast},A)\geqslant 2(|A|+|C|)-O(\varepsilon n).$
			In particular, this implies
			\begin{equation}\label{EQ-d1}
				d^-(d^{\ast\ast},C)\geqslant |C|-O(\varepsilon n) \mbox{ and } d^+(d^{\ast\ast},A)\geqslant |A|-O(\varepsilon n).
			\end{equation}

			Let $a^{\ast}$ be any vertex of $N^+(d^{\ast\ast},A)\backslash A^L$. Since $a^{\ast}\notin A^L$ and $d^{\ast\ast}\notin D^L$, we have 
			\begin{equation}\label{EQ-d}
				d^+(a^{\ast})+d^-(d^{\ast\ast})\leqslant 2\varepsilon n+ d^+(a^{\ast},A)+d^+(a^{\ast},B)+d^-(d^{\ast\ast},B)+d^-(d^{\ast\ast},C).
			\end{equation}
			
			Next, we claim that $d^{\ast\ast}$ has at least $|A|/2-5\varepsilon^{2/3}n$ in-neighbors in $B$.
			If $|A|/2<5\varepsilon^{2/3}n$, the claim holds trivially. Thus, we can assume that $|A|$ (and consequently $|C|$) is large. In particular,  (\ref{EQ-d1}) and \ref{Y1} show that $|N^+(d^{\ast\ast},A)\backslash A^L|$ is large. Let $a^{\ast\ast}$ be a vertex in $N^+(d^{\ast\ast},A)\backslash A^L$ with minimum out-degree in  $G[N^+(d^{\ast\ast},A)\backslash A^L]$. Since $G$ is oriented, we have $d^+(a^{\ast\ast},N^+(d^{\ast\ast},A)\backslash A^L)\leqslant |N^+(d^{\ast\ast},A)\backslash A^L|/2$. By (\ref{EQ-d1}) and  $|A^L|\leqslant 2\varepsilon^{2/3} n$, we have $|A|\leqslant |N^+(d^{\ast\ast},A)\backslash A^L|+3\varepsilon^{2/3} n$ and hence $d^+(a^{\ast\ast},A)\leqslant |A|/2+3\varepsilon^{2/3} n$. Then (\ref{EQ-d}) implies that $d^+(a^{\ast\ast})+d^-(d^{\ast\ast})\leqslant 2\varepsilon n+ |A|/2+3\varepsilon^{2/3} n+|B|+d^-(d^{\ast\ast},B)+|C|.$ 
			
			On the other hand, since $a^{\ast\ast} \in N^+(d^{\ast\ast}, A)$ and $G$ is oriented, $a^{\ast\ast}d^{\ast\ast}\notin E(G)$.
			Thus, the Ore-type condition gives that $d^+(a^{\ast\ast})+d^-(d^{\ast\ast})\geqslant 3n/4-O(\varepsilon n).$  
			Combining the upper and lower bounds of $d^+(a^{\ast\ast})+d^-(d^{\ast\ast})$, we have 
			\begin{equation}\label{EQ-d2}
				d^-(d^{\ast\ast},B)\geqslant |A|/2-5\varepsilon^{2/3}n.
			\end{equation}
			
			By symmetry, we also obtain the following inequality:  
			\begin{equation}\label{EQ-d3}
				d^+(d^{\ast\ast}, B) \geqslant |C|/2 - 5\varepsilon^{2/3}n.
			\end{equation}
			To see this, let $c^{\ast\ast}$ be a vertex in $N^-(d^{\ast\ast}, C) \backslash  C^L$ with minimum in-degree in $G[N^-(d^{\ast\ast}, C) \backslash  C^L]$. A similar argument as above shows that 
			$d^-(c^{\ast\ast}, C) \leqslant |C|/2 + 3\varepsilon^{2/3} n.$ Since $d^{\ast\ast}c^{\ast\ast} \notin E(G)$, the Ore-type condition together with \ref{Y2}-\ref{Y3} implies  
			$2\varepsilon n + |C|/2 + 3\varepsilon^{2/3} n + |B| + |A| + d^+(d^{\ast\ast}, B) \geqslant 3n/4 - O(\varepsilon n),$ from which (\ref{EQ-d3}) follows.
			
			By inequalities (\ref{EQ-d1}), (\ref{EQ-d2}), (\ref{EQ-d3}) and the fact that $\varepsilon\ll \delta$, we get that  $d^{\ast\ast}$ is $\delta$-good. Since there are at least $|D|-3\varepsilon^{2/3}n$ choices for $d^{\ast\ast}$, we conclude that $D$ has at most $\delta n/4$ $\delta$-bad vertices.
		\end{proof} 
		
		\begin{claim}
			There are at most $\delta n/4$ $\delta$-bad vertices in $A$ and $e(A,B)\geqslant |A||B|-O(\varepsilon^{1/6}n^2)$.
		\end{claim}
		
		\begin{proof}
			If $|A|\leqslant \delta n/4$, the claim holds trivially. Thus, we can assume $|A|\geqslant \delta n/4$, which implies $|C|\geqslant \delta n/4$ as well. First, observe that by (\ref{EQ-d3}) and the fact that $G$ is oriented, we have $d^-(d^{\ast\ast},B)\leqslant |B|-|C|/2+5\varepsilon^{2/3}n$. Substituting this into (\ref{EQ-d}) yields 
			$d^+(a^{\ast})+d^-(d^{\ast\ast})\leqslant 2\varepsilon n+ d^+(a^{\ast},A)+2|B|-|C|/2+5\varepsilon^{2/3}n+|C|.$  
			Since there is no arc from $a^{\ast}$ to $d^{\ast\ast}$, 
			the Ore-type condition gives that 
			$d^+(a^{\ast})+d^-(d^{\ast\ast})\geqslant 3n/4-O(\varepsilon n).$  Combining the upper and lower bounds and the fact $|A| + |C| = n/2 \pm O(\varepsilon n) $ and $|B|, |D| = n/4 \pm O(\varepsilon n)$, we have 
			\begin{equation}\label{EQ-a1}
				d^+(a^{\ast},A)\geqslant |A|/2-6\varepsilon^{2/3}n \mbox{ and thus } d^-(a^{\ast},A)\leqslant |A|/2+6\varepsilon^{2/3}n.
			\end{equation}
			
			Recall that $a^{\ast}$ is an arbitrary vertex in $N^+(d^{\ast\ast},A)\backslash A^L$ and $|N^+(d^{\ast\ast},A)\backslash A^L|\geqslant |A|-O(\varepsilon^{2/3}n)$ by (\ref{EQ-d1}) and \ref{Y1}. Thus there are at most $O(\varepsilon^{2/3}n)$ vertices in $A$ having out-degree less than $|A|/2-6\varepsilon^{2/3}n$ in $A$. Moreover, it follows that $e(A)\geqslant |A|^2/2-O(\varepsilon^{2/3}n^2)$.
			
			Next, we claim that there are at most $\varepsilon^{1/3}n$   vertices in $A$ having in-degree less than $|A|/2-\varepsilon^{1/4}n$ in $A$. Suppose to the contrary that there are more than $\varepsilon^{1/3}n$ such vertices. Then the second statement of inequality (\ref{EQ-a1}) implies that 
			\begin{align*}
				e(A)&\leqslant \varepsilon^{1/3}n(|A|/2-\varepsilon^{1/4}n)+ |A^L||A|+(|A\backslash A^L|-\varepsilon^{1/3}n)(|A|/2+6\varepsilon^{2/3}n)\\
				&\leqslant |A|^2/2-\varepsilon^{7/12}n^2+8\varepsilon^{2/3}n^2,
			\end{align*}
			which contradicts the fact that $e(A)\geqslant |A|^2/2-O(\varepsilon^{2/3}n^2)$.  Therefore, there are at most $2\varepsilon^{1/3}n$ vertices in $A$ have either out-degree or in-degree less than $|A|/2-\varepsilon^{1/4}n$ in $G[A]$. Let $A_1$ be the set of these vertices. Clearly, $|A_1|\leqslant 2\varepsilon^{1/3}n$.  Observe that since $G$ is oriented, for each $a\in A\backslash A_1$, we have  $d^-(a, A), d^+(a, A)= |A|/2\pm \varepsilon^{1/4}n$.  
			
			Recall that there are at least $|D|-3\varepsilon^{2/3}n$ choices for the vertex $d^{\ast\ast}$. Then by  (\ref{EQ-d1}), we get $e(D,A)\geqslant |A||D|-O(\varepsilon^{2/3}n^2)$. Due to Lemma \ref{LEM:edges-to-degree} (i), there are at most $\varepsilon^{1/6} n$ vertices of $A$  have in-degree less than  $|D|-\varepsilon^{1/3}n$ in $D$. Let $A_2$ be the set of those vertices with small in-degree. Clearly, $|A_2|\leqslant\varepsilon^{1/6} n$. Let $a$ be any vertex of $N^+(d^{\ast\ast},A)\backslash (A^L\cup A_1\cup A_2)$. Recall that $d^-(d^{\ast\ast},B)\leqslant |B|-|C|/2+5\varepsilon^{2/3}n$  by inequality (\ref{EQ-d3}). Since $a\notin A_1$, we have $d^+(a,A)=|A|/2\pm \varepsilon^{1/4}n$. Then inequality (\ref{EQ-d}) implies that $d^+(a)+d^-(d^{\ast\ast})\leqslant 2\varepsilon n+ |A|/2+\varepsilon^{1/4}n+d^+(a,B)+|B|-|C|/2+5\varepsilon^{2/3}n+|C|.$ 
			On the other hand, the Ore-type condition implies $d^+(a)+d^-(d^{\ast\ast})\geqslant 3n/4-O(\varepsilon n).$ 
			Comparing these bounds, we obtain $d^+(a,B)\geqslant |B|-O(\varepsilon^{1/4}n)$. Hence, $a$ is a $\delta$-good vertex as $\varepsilon\ll \delta$ and  $a\notin A_1\cup A_2$. 
			
			Now we count the number of such good vertices. It follows by   (\ref{EQ-d1}) that $|N^+(d^{\ast\ast},A)|\geqslant|A|-O(\varepsilon n)$. Moreover, since $|A_1|\leqslant 2\varepsilon^{1/3}n$,  $|A_2|\leqslant\varepsilon^{1/6} n$ and $|A^L|= 2\varepsilon^{2/3}n$,  there are at least $|A|-2\varepsilon^{1/6} n$  choices for $a$. This means that $A$ has at most $2\varepsilon^{1/6} n<\delta n/4$ $\delta$-bad vertices as $\varepsilon \ll \delta$. Moreover, we have 
			$e(A,B)\geqslant (|A|-2\varepsilon^{1/6} n)(|B|-O(\varepsilon^{1/4}n))\geqslant |A||B|-O(\varepsilon^{1/6}n^2).$
		\end{proof}
		
		By symmetry, we can obtain that $C$ has at most $\delta n/4$ $\delta$-bad vertices and $e(B,C)\geqslant |B||C|-O(\varepsilon^{1/6}n^2)$. This is achieved by swapping the roles of $A$ and $C$, $B$ and $D$, and  interchanging ``$+$'' and ``$-$'' in the proof of the preceding claim. Applying Lemma \ref{LEM:edges-to-degree} (i) to $e(A,B)$ and $e(B,C)$ respectively, we obtain that $B$ has at most $\delta n/8$ vertices with in-degree less than $|A|-\delta n$ in $A$ and at most $\delta n/8$ vertices with out-degree less than $|C|-\delta n$ in $C$. It follows that $B$ has at most $\delta n/4$ bad vertices in total.  Therefore, the partition $(A,B,C,D)$ contains at most $\delta n$ $\delta$-bad vertices, which completes the proof of Lemma \ref{LEM-findbad}.
	\end{proof}

	Let $(A,B,C,D)$ be a partition of $G$. An edge $e$ is called \emph{special} for  $(A,B,C,D)$ if it belongs to $E(A\cup D,C\cup D)\cup E(B\cup C,A\cup B)$. Special edges are useful for embedding a long antidirected cycle in $G$. The following lemma shows that the degree condition $\sigma_{+-}(G) \geqslant (3n+2)/4$ guarantees the existence of two vertex-disjoint special edges for the nice partition.

	\begin{lemma}\label{LEM:special-edges}
		Let $0<1/n\ll\delta  \ll 1$, and let $G$ be an $n$-vertex oriented graph with a $\delta$-nice partition $(A,B,C,D)$. If $\sigma_{+-}(G)\geqslant  (3n+2)/4$, then there are at least two vertex-disjoint special edges for the partition.
	\end{lemma}
	\begin{proof}
		Suppose to the contrary that there are no two vertex-disjoint special edges for $(A,B,C,D)$. Since $G$ is an oriented graph, only the following two cases happen.

		\medskip
		\textrm{\textbf{Case 1: There are exactly three special edges in $G$, and they form a triangle.}}
		\medskip
		
		Let $T$ be the triangle induced by the three special edges. It should be mentioned that $T$ possibly is not directed. Observe that all edges in $G[B]$ are special and $|B|= n/4 \pm O(\delta n)$ by \ref{NP2}. Thus, we can choose two vertices $b_1,b_2$ in $B\backslash V(T)$ such that $b_1b_2\notin E(G)$ and $b_2b_1 \notin E(G)$. By the Ore-type condition, we have $$d(b_1)+d(b_2)=d^+(b_1)+d^-(b_1)+d^+(b_2)+d^-(b_2)\geqslant (3n+2)/2.$$ On the other hand, since all special edges are incident to vertices in $V(T)$ and $b_1,b_2\notin V(T)$, both $b_1$ and $b_2$ have in-degree and out-degree zero in $B$. This implies that $d(b_1)+d(b_2)\leqslant   2(|A|+|C|+|D|)$.  Combine the upper and lower bounds for $d(b_1)+d(b_2)$, we have $|B|=n-(|A|+|C|+|D|)\leqslant  (n-2)/4$. Similarly, $|D|\leqslant  (n-2)/4$.

		Since all edges in $G[A]$ and $G[C]$ are not special, we have $|V(T)\cap A|\leqslant 1$ and $|V(T)\cap C| \leqslant  1$. It follows by \ref{NP1} and \ref{NP2} that both $|B\backslash V(T)|$, $|C\backslash V(T)|$ and $|D\backslash V(T)|$ are larger than one. Let $b\in B,c\in C,d\in D$ be three vertices not in $V(T)$.  First we consider the case  that $|A|\geqslant 1$ and assume $a\in A$.  Since the edges in $E(T)$ are the only special edges,  we have $ac,ca\notin E(G)$. 
		The Ore-type degree condition implies that $d(a)+d(c) \geqslant (3n+2)/2$. On the other hand, we have $d(a)+d(c) \leqslant (|A|-1+|B|+|D|+ |V(T) \cap C|)+(|C|-1+|B|+|D|)$. Combining these bounds yields $|B|+|D|\geqslant n/2+2$, which contradicts the fact that $|B|+|D|\leqslant n/2-1$.
		
		So it suffices to consider the case that $A=\emptyset$. Note that as $b,c,d\notin V(T)$, there is no edge from $d$ to $C\cup D$ and no edge from $B\cup C$ to $b$. In particular, we get that $dc, cb\notin E(G)$. Again, by the Ore-type degree condition, we have $d^+(d)+d(c)+d^-(b)\geqslant (3n+2)/2$. Meanwhile, $d^+(d)+d(c)+d^-(b) \leqslant  |B| + (|B|+|D|+|C|-1)+|D|$. Combining the two bounds again, we have $|B|+|D|\geqslant n/2+2$, which contradicts the fact that $|B|+|D|\leqslant n/2-1$. 
		
		\medskip
		\textrm{\textbf{Case 2: All special edges are incident to a common vertex, or there are no special edges in $G$.}}
		\medskip
		
		If special edges exist, let $v$  incident to all of them, and define $X=\{v\}$. Otherwise, set $X=\emptyset$. Remove $X$ from the original partition, and denote the resulting subsets again by $A,B,C$ and $D$. Clearly, $n=|A|+|B|+|C|+|D|+|X|$.
		
		First, suppose  $|A|\geqslant 1$. Choose one vertex from each (updated) set: $a \in A$, $b \in B$, $c \in C$, and $d \in D$.  
		Since all special edges are incident to the vertex in $X$, none of $ad,dc,cb$ and $ba$ is an edge of $G$.
		Then the Ore-type condition implies that  $$3n+2\leqslant  d(a)+d(b)+d(c)+d(d) <  3(|A|+|B|+|C|+|D|) +4|X|-2,$$ which contradicts the fact that $|X|\leqslant 1$.
		
		It remains to consider the case $|A|=0$. As argued in the first paragraph of \textrm{\textbf{Case 1}}, we can obtain $|B|+|D|\leqslant n/2-1$. Choose vertices $d \in D$, $b \in B$ and $c \in C$. Again, as all special edges are incident to the vertex in $X$, neither $dc$ nor $cb$ is an edge of $G$. Then the Ore-type condition implies that  $$(3n+2)/2\leqslant  d^+(d)+d(c)+d^-(b) \leqslant |B|+(|C|-1+|B|+|D|)+|D|+3|X|.$$ Then $n/2-1\geqslant|B|+|D|\geqslant n/2+2-2|X|$, which contradicts the fact that $|X| \leqslant  1$. 
		
		Therefore, there are two vertex-disjoint special edges for $(A,B,C,D)$, which completes the proof of Lemma \ref{LEM:special-edges}.
	\end{proof}

	In order to construct the desired antidirected cycle, we now introduce two key definitions that will be essential for our subsequent arguments. The first formalizes the notion of a vertex being well-connected to specific subsets of vertices, while the second  defines a particular type of path that will serve as a building block in our cycle construction.
	
	\begin{definition} \emph{(Acceptable vertex)}\label{DEF-accevtx}
		Let $U_1, U_2 \subseteq V(G)$ and $v \in V(G)$. If $d^-(v,U_1)\geqslant n/100$  and $d^+(v,U_2)\geqslant n/100$, then $v$ is said to be acceptable for $(U_1,U_2)$ and denote it by $v\in \mathcal{A}(U_1,U_2)$.
	\end{definition}

	\begin{definition} \emph{(Proper path)}\label{DEF-goodpath}
		Let $\varepsilon \ll \delta<1$ and let $(A,B,C,D)$ be an $\varepsilon$-nice partition of $V(G)$. A path $P=v_1v_2\cdots v_d$ is called a $D$-proper path if it satisfies the following:
		\begin{enumerate}[label =\upshape \textbf{(P\arabic{enumi})}, ref =\upshape (P\arabic{enumi})]
			\setlength{\itemindent}{1.5em}
			\item $P$ is an antidirected path of order at most 10;\label{P1}
			\item  $v_1v_2, v_{d-1}v_d\in E(G)$;\label{P2}
			\item the end-vertices $v_1,v_d$ are $\delta$-good vertices of $D$.\label{P3}
		\end{enumerate}
	\end{definition} 
	
	Note that every $D$-proper path $P=v_1v_2\cdots v_d$ must have even order since it is an antidirected path with $v_1v_2, v_{d-1}v_d\in E(G)$. 
	
	With these definitions established, we now state and prove the following lemma, which guarantees the existence of two vertex-disjoint proper paths under the given conditions.
	
	\begin{lemma}\label{LEM-goodpath}
		Let $n$ be an even positive integer and let $0<1/n\ll \varepsilon \ll 1$. Suppose  $G$ is an $n$-vertex oriented graph with $\sigma_{+-}(G)\geqslant (3n+2)/4$. Then for every $\varepsilon$-nice partition $(A,B,C,D)$ of $V(G)$, there are two vertex-disjoint $D$-proper paths in $G$.
	\end{lemma}
	\begin{proof} By Lemma \ref{LEM-findbad}, there exists a constant $\delta \gg \varepsilon$ such that $G$ contains at most $\delta n$ $\delta$-bad vertices. Reassign each bad vertex $v$ according to the following rules:
		
		\settasks{label= \textbullet} 
		\begin{tasks}(2)
			\task Assign $v$ to $A$ if $v \in \mathcal{A}(A \cup D, A\cup B)$;
			\task Assign $v$ to $B$ if $v \in \mathcal{A}(A \cup D, C \cup D)$;
			\task Assign $v$ to $C$ if $v \in \mathcal{A}(B\cup C, C \cup D)$;
			\task Assign $v$ to $D$ if $v \in \mathcal{A}(B\cup C, A\cup B)$.
		\end{tasks}

		It is not difficult to check that the reassignment is valid as $\delta^0(G) \geqslant n/6$ by Proposition \ref{PROP:Ore-implies-semideg}. For simplicity, we continue to denote the resulting partition as $(A,B,C,D)$. By symmetry, we may assume $|A|\leqslant|C|$; otherwise, reverse all edges of  $G$ and swap the labels of $A$ and $C$. 
		So the new partition remains $\delta$-nice. Moreover, by Lemma \ref{LEM:special-edges}, $G$ contains two vertex-disjoint special edges for $(A,B,C,D)$.

		We next show that every special edge can be extended to a $D$-proper path while avoiding a fixed set $W$ with constant size. The claim then follows by constructing such paths for each of the two vertex-disjoint special edges, ensuring that each path avoids the vertices used in the other. Note that if $P$ is a $D$-proper path for the new partition $(A,B,C,D)$, then it remains $D$-proper for the original $\varepsilon$-nice partition, since the end-vertices of $P$ are $\delta$-good vertices and thus belong to the original set $D$.

		Let $uv$ be any special edge and let $W\subseteq V(G)\backslash\{u,v\}$ be any set with $|W|\leqslant 20$. Next we show that $uv$ can be extended to a $D$-proper path avoiding all vertices of $W$. First, suppose $uv \in E(B\cup C, A\cup B)$. By Definition \ref{DEF-accevtx} and the reassignment of bad vertices, we have 
		\[
		d^+(u, C \cup D) \geqslant n/100 \quad \text{and} \quad d^-(v, A \cup D) \geqslant n/100.
		\]
		
		Since there are at most $\delta n$ $\delta$-bad vertices and $|W|\leqslant 20$, we can choose a $\delta$-good vertex $u_1$ in $N^+(u,C\cup D)\backslash W$ and a $\delta$-good vertex $v_1$ in $N^-(v,A \cup D)\backslash W$. By  \ref{GA}-\ref{GD}, $u_1$ has a $\delta$-good in-neighbor $u_2\in C\backslash W$ and $u_2$ has a $\delta$-good out-neighbor $u_3\in D\backslash W$. Similarly, if  $|A|\geqslant n/200$, then $v_1$ has a $\delta$-good out-neighbor in  $v_2\in A\backslash W$ and $v_2$ has a $\delta$-good in-neighbor $v_3\in D\backslash W$. Then the path $v_3v_2v_1vuu_1u_2u_3$ is the desired $D$-proper path. For the case that $|A|<n/200$, we have $d^-(v,D)\geqslant n/200$ as  $d^-(v, A \cup D) \geqslant n/100$. Then we may assume that the vertex $v_1$ belongs to $D\backslash W$ and thus  $v_1vuu_1u_2u_3$ is the desired $D$-proper path.

		The case $uv\in E(A\cup D, C\cup D)$ can be handled similarly, and we omit the details.
	\end{proof}

	Building upon the previous results, we now present a more technical lemma that describes how to construct an antidirected path incorporating the proper paths while satisfying several important properties. This lemma will be instrumental in the final stage of our proof.

	\begin{lemma}\label{LEM-ACpath}
		Let $n$ be an even positive integer and let $0<1/n\ll \varepsilon \ll 1$. Then there exists $\delta \gg \varepsilon$ such that the following holds. Suppose that $G$ is an $n$-vertex oriented graph with $\sigma_{+-}(G)\geqslant 3n/4-O(\varepsilon n)$ and $(A,B,C,D)$ is an $\varepsilon$-nice partition of $G$. If $G$ contains two vertex-disjoint $D$-proper paths $P_1$ and $P_2$, then there exists an antidirected path $P$ satisfying the following:
		\begin{enumerate}[label =\upshape \textbf{(L\arabic{enumi})}, ref =\upshape (L\arabic{enumi})]
			\setlength{\itemindent}{1.5em}
			\item $A\cup V(P_1)\cup V(P_2)\subseteq V(P)$ and $P$ contains all $\delta$-bad vertices of $G$; \label{L1}
			\item both end-vertices of $P$ are $\delta$-good vertices in $D$ and they are sinks of $P$; \label{L2}
			\item $|C\backslash V(P)| > |B\backslash V(P)|+ |D\backslash V(P)|+n/300$; \label{L3}
			\item $|C\cap V(P)|= O(\delta n)$. \label{L4}
		\end{enumerate}
	\end{lemma}
	\begin{proof}

		To ensure the condition \ref{L3}, 
		we first construct a long antidirected path $P^{\ast}$ between $B$ and $D$. Let $H$ be the underlying bipartite graph with bipartition $(B,D)$ and edge set $E(D,B)$. Since $(A,B,C,D)$ is an $\varepsilon$-nice partition, the property \ref{NP1} guarantees that $|C|\geqslant n/4-O(\varepsilon n)$. By Lemma \ref{LEM-findbad}, $G$ contains at most $\delta n$ $\delta$-bad vertices. Combining this with \ref{GD}, we have 
		$$e(H)=e(D,B)\geqslant (|D|-\delta n)(|C|/2-\delta n)\geqslant n^2/32 -O(\delta n^2).$$ 
		Since $V(H)=|B\cup D|=n/2\pm O(\varepsilon n)$, Proposition \ref{PRO-findsubgraph} implies that there exists  a subgraph $H^\prime \subseteq H$ with $\delta(H^\prime) \geqslant n/100$. Moreover, since $G$ has at most  $\delta n$ $\delta$-bad vertices, $H^\prime$ contains a path of order $n/200$ whose end-vertices are in $D$ and all vertices are $\delta$-good. This path clearly corresponds to the desired antidirected path $P^{\ast}$.

		Set $\mathcal{P}=V(P_1\cup P_2\cup P^{\ast})$.  Clearly, $|\mathcal{P}|\leqslant n/100$ by \ref{P1}.  Next, we reassign each $\delta$-bad vertex $v$ of $V(G)\backslash \mathcal{P}$ to one of the sets $\mathcal{B}_A$ and $\mathcal{B}_C$ according to the following criteria:

		$\bullet$ Assign $v$ to $\mathcal{B}_A$ if one of $d^+(v,A)$, $d^-(v,A)$ and $d^-(v,D)$ exceeds $n/50$; 
		
		$\bullet$ Assign $v$ to $\mathcal{B}_C$ if  one of $d^-(v,C)$, $d^+(v,C)$ and $d^+(v,D)$ exceeds $n/50$.

		This assignment is valid. Indeed, by Proposition \ref{PROP:Ore-implies-semideg}, we have $\delta^0(G)\geqslant n/6$. If a bad vertex $v$ could not be assigned to $\mathcal{B}_A\cup \mathcal{B}_C$, then we would have $d(v,B)\geqslant 2(n/6-3n/50)>|B|$, which contradicts the fact that $G$ is oriented.  Remove all vertices of $\mathcal{B}_A\cup\mathcal{B}_C$ from the partition and continue to denote the resulting sets as $A,B,C$ and $D$. Observe that now all  vertices in $A\cup B \cup C\cup D$ are $\delta$-good except for the internal vertices in $P_1\cup P_2$. Note that $P_1$ and $P_2$ are still $D$-proper paths since all end-vertices are $\delta$-good vertices in $D$.

		Let $(X,Y)\in\{(A,B),(B,C),(C,D),(D,A)\}$. It follows by \ref{GA}-\ref{GD} that $d^+(x,Y)\geqslant |Y|-\delta n$ and $d^-(y,X)\geqslant |X|-\delta n$ for any $x\in X,y\in Y$. Therefore, we have 
		
		\medskip
		\textbf{(XY)} For any $x_1,x_2\in X$ and $y_1,y_2\in Y$ with $(X,Y)\in\{(A,B),(B,C),(C,D),(D,A)\}$, we have $|N^+(x_1,Y)\cap N^+(x_2,Y)|\geqslant |Y|-2\delta n$ and $|N^-(y_1,X)\cap N^-(y_2,X)|\geqslant |X|-2\delta n$. 
		\medskip
		
		Similarly, we have the following statements  by \ref{GA}-\ref{GD} and $|\mathcal{B}_A\cup\mathcal{B}_C|\leqslant \delta n$.

		\textbf{(CD)} $|N^-(c,C)\cap N^-(d,C)|\geqslant |C|/2-4\delta n$ for any $c\in C,d\in D$. 
		
		\textbf{(DD)}  $|N^+(d_1,B)\cap N^+(d_2,B)|\geqslant |C|-|B|-4\delta n$ for any $d_1,d_2\in D$. 
		
		\textbf{(AD)}  $|N^+(a,A)\cap N^+(d,A)|\geqslant |A|/2-4\delta n$ and $|N^+(a,B)\cap N^+(d,B)|\geqslant |C|/2-4\delta n$  for any $a\in A,d\in D$. 
		
		\medskip
		
		Now we construct a (short) path $P_C$ in $G- \mathcal{P}$ covering all vertices of $\mathcal{B}_C$. Let $\mathcal{B}_C^1, \mathcal{B}_C^2$ and $\mathcal{B}_C^3$ be disjoint sets of vertices in $\mathcal{B}_C$ with $d^+(v,D)>n/50$, $d^+(v,C)>n/50$ and  $d^-(v,C)>n/50$, respectively. 
		For each $v\in \mathcal{B}_C^1$, pick two out-neighbors of $v$ in $D\backslash \mathcal{P}$. Moreover, as $|\mathcal{P}|\leqslant n/100$, one may assume that those out-neighbors are pairwise distinct. Applying  (XY) with $(X,Y)=(C,D)$, every pair of vertices in $D$ has at least $|C|-2\delta n\geqslant n/5$ common in-neighbors in $C\backslash \mathcal{P}$ as   $|C|\geqslant n/4-O(\varepsilon n)$. Repeat the above progress, one may get an antidirected path $L_1$ in $G-\mathcal{P}$ covering all vertices of $\mathcal{B}_C^{1}$ with form 
		$d_1C(D\mathcal{B}_C^{1}DC)^{|\mathcal{B}_C^{1}|}d_2$, where $d_1,d_2 \in D$. Clearly $d_1,d_2$ are sink vertices of the path $L_1$.
		
		In the same way, applying  (XY) with $(B,C)$ and $(C,D)$ respectively, $G-\mathcal{P}$ has two disjoint antidirected paths  
		$L_2=c_1B(C\mathcal{B}_C^{2}CB)^{|\mathcal{B}_C^{2}|}c_2$ and 
		$L_3=c_3D(C\mathcal{B}_C^{3}CD)^{|\mathcal{B}_C^{3}|}c_4$, where $c_i \in C$ for each $i\in[4]$.  Clearly, $c_1,c_2$ are sink vertices of $L_2$ and  $c_3,c_4$ are source vertices of $L_3$. For $c_i$ with $i=3,4$,  pick an unused out-neighbor $d_i$ of $c_i$ in $D$ according to  \ref{GC}. By (CD),  $c_1$ and $d_3$ (resp., $c_2$ and $d_2$) have more than $n/10$ common in-neighbors in $C$. Then we get the desired path $P_C=(d_1)L_1CL_2CDL_3d_4$ with $|P_C|=O(\delta n)$ as $|\mathcal{B}_C|\leqslant \delta n$. Note that the end-vertices of $P_C$ are sink vertices in $D$.

		\medskip
		We next construct an antidirected path $P_A$ in the remaining $G$ to cover all vertices of $A\cup \mathcal{B}_A \cup \mathcal{P}$. In the following, we always assume that the selected vertex is unused. This is possible since  $|P_1|,|P_2|,|P_C|,|\mathcal{B}_A|,|\mathcal{B}_C|=O(\delta n), |P^{\ast}|=n/200$ and  the number of  common neighbors are large, as stated in  (XY),  (CD), (DD) and (AD). 
		
		First noted that (XY) shows that every pair of vertices of $D$ has at least $|A|-2\delta n$ common out-neighbors in $A$. Hence we do not have enough vertices  to connect two vertices of $D$ when $|A|$ is small.  However, if $|A|$ is small, say $|A|\leqslant n/50$, then \ref{NP1}-\ref{NP2} and (DD) imply that every two vertices of $D$ has at least $|B|-n/50-O(\delta n)$ common out-neighbors in $B$. So next we construct $P_A$ by considering the following two cases based on the size of $A$.

		\medskip
		\textrm{\textbf{Case 1: $|A|\leqslant n/50$.}}
		\medskip
		
		\begin{figure}
			
			\subfigure{
				\begin{minipage}[t]{0.49\linewidth}
					\centering
					\begin{tikzpicture}[black,line width=1pt,scale=0.7]
						\path (-6,-3) (6,3); 
						
						\draw[draw=black] (3,-3) rectangle (4,3);
						\draw[draw=black] (-3,-3) rectangle (-4,3);
						\draw[draw=black] (-1,3.25) rectangle (1,4.25);

						\foreach \i/\j in {(3.5,1.2)/b1,(3.5,0.6)/b9,(3.5,0)/b2,(3.5,-0.6)/b3,(3.5,2.5)/b4,(3.25,1.75)/b5,(3.5,-2.5)/b8,(3.5,-1.9)/b7,(3.5,-1.2)/b6,(0.75,3.9)/a1,(0.425,3.65)/a2,(0.1,3.4)/a3,(-0.25,3.4)/a4,(-0.65,3.9)/a5,(-3.5,1.7)/d14,(-3.5,1.3)/d1,(-3.5,0.9)/d2,(-3.5,0.5)/d3,(-3.5,0.1)/d4,(-3.5,-0.3)/d5,(-3.5,-0.8)/d6,(-3.5,-1.3)/d7,(-3.5,-1.8)/d8,(-3.5,-2.3)/d9,(-3.5,-2.8)/d10,(-3.6,2.75)/d11,(-3.4,2.5)/d12,(-3.2,2.25)/d13}{\filldraw[black]\i circle (1.5pt)coordinate(\j);}
						
						\foreach \i/\j in {(-4.5,1.1)/u1,(-4.5,0.3)/u2}{
							\node at \i [rectangle,draw,inner sep=0.7mm][](\j){};}
						
						\coordinate (u3) at (-5,-0.3);
						\coordinate (u6) at (-5,-0.8);
						\coordinate (u4) at (-5,-1.55);
						\coordinate (u5) at (-5,-2.55);
						\node at (-5.5,-0.55) [black]{$P^\ast$};
						\node at (-5.5,-1.55) [black]{$P_1$};
						\node at (-5.5,-2.55) [black]{$P_2$};
						\node at (0,2.5) [black]{$A$};
						\node at (3.5,-3.3) [black]{$B$};
						\node at (-3.5,-3.3) [black]{$D$};
						
						\node at (2.4,3.9) [black]{$P_{AB}$};
						\node at (-5,0.7) [black]{$Q$};
						\node at (-2.4,3.9) [black]{$P_{DA}$};
						
						\foreach \i/\j in {d11/a5,d12/a5,d12/a4,d13/a4,a1/b4,a2/b4,a2/b5,a3/b5,d1/u1,d2/u1,d3/u2,d4/u2}{\draw[leftlearrow={latex}, black]( \i) -- (\j);}
						
						\foreach \i/\j/\k/\l in {d5/170/170/u3,d7/180/85/u4,u4/-85/180/d8,d9/180/85/u5,u5/-85/180/d10}{\draw [leftlearrow={latex},black] (\i) to [out=\j,in=\k]  (\l);}
						\draw [leftlearrow={latex},black] (d6) to [out=-170,in=-170]  (-4.99,-0.782);
						\draw[decoration={aspect=-0.05, segment length=1.38mm, amplitude=0.8mm,coil},decorate,line width=1pt] (u3) -- (u6);
						\foreach \i/\j in {d14/b2,d14/b9,d13/b1,a3/b1,d1/b2,d2/b3,d3/b3,d4/b6,d5/b6,d6/b7,d7/b7,d9/b8}{\draw[middlearrow={latex}, red,dashed]( \i) -- (\j);}
						\draw [middlearrow={latex},red,dashed] (a1) to [out=0,in=90]  (4.7,1) to  [out=-90,in=45] (b8);
						\draw [middlearrow={latex},red,dashed] (d11) to [out=180,in=160]  (-3.8,2.25) to  [out=-20,in=160] (b9);
						\foreach \i/\j in {(3.5,1.2)/b1,(3.5,0.6)/b9,(3.5,0)/b2,(3.5,-0.6)/b3,(3.5,2.5)/b4,(3.25,1.75)/b5,(3.5,-2.5)/b8,(3.5,-1.9)/b7,(3.5,-1.2)/b6,(0.75,3.9)/a1,(0.425,3.65)/a2,(0.1,3.4)/a3,(-0.25,3.4)/a4,(-0.65,3.9)/a5,(-3.5,1.7)/d14,(-3.5,1.3)/d1,(-3.5,0.9)/d2,(-3.5,0.5)/d3,(-3.5,0.1)/d4,(-3.5,-0.3)/d5,(-3.5,-0.8)/d6,(-3.5,-1.3)/d7,(-3.5,-1.8)/d8,(-3.5,-2.3)/d9,(-3.5,-2.8)/d10,(-3.6,2.75)/d11,(-3.4,2.5)/d12,(-3.2,2.25)/d13}{\filldraw[black]\i circle (1.5pt)coordinate(\j);}
					
					\end{tikzpicture}\caption*{(a) $|A|$ is small}
			\end{minipage}}
			\subfigure{
				\begin{minipage}[t]{0.49\linewidth}
					\centering
					\begin{tikzpicture}[black,line width=1pt,scale=0.7]
						\path (-6,3) (6,-3); 
						\draw[draw=black] (3,-3) rectangle (4,2);
						\draw[draw=black] (-3,-3) rectangle (-4,2);
						\draw[draw=black] (-2,3) rectangle (2,4.5);

						\foreach \i/\j in {(3.5,0)/b2,(3.5,-1)/b3,(1.85,3.15)/a1,(1.85,3.5)/a2,(1.85,4.35)/a3,(1.85,4)/a4,(1.5,3.75)/a5,(-1.85,3.15)/a6,(-1.85,3.5)/a7,(-1.85,4.35)/a8,(-1.85,4)/a9,(-1.5,3.75)/a10,(0.6,3.15)/a11,(1,3.15)/a12,(1.4,3.15)/a13,(-0.35,3.15)/a14,(1,4.35)/a15,(-1.1,3.55)/a16,(0.2,3.15)/a17,(-3.5,1.7)/d14,(-3.5,1.3)/d1,(-3.5,0.9)/d2,(-3.5,0.5)/d3,(-3.5,0.1)/d4,(-3.5,-0.3)/d5,(-3.5,-0.8)/d6,(-3.5,-1.3)/d7,(-3.5,-1.8)/d8,(-3.5,-2.3)/d9,(-3.5,-2.8)/d10}{\filldraw[black]\i circle (1.5pt)coordinate(\j);}
						
						\foreach \i/\j in {(-4.5,1.1)/u1,(-4.5,0.3)/u2,(-3,3.325)/d11,(-3,4.175)/d12,(3,3.325)/b4,(3,4.175)/b5}{
							\node at \i [rectangle,draw,inner sep=0.7mm][](\j){};}

						\coordinate (u3) at (-5,-0.3);
						\coordinate (u6) at (-5,-0.8);
						\coordinate (u4) at (-5,-1.55);
						\coordinate (u5) at (-5,-2.55);
						\node at (-5.5,-0.55) [black]{$P^\ast$};
						\node at (1,4.05) [black]{$v_1$};
						\node at (-0.6,3.55) [black] {$v_2$};
						\node at (-5.5,-1.55) [black]{$P_1$};
						\node at (-5.5,-2.55) [black]{$P_2$};
						\node at (-5,0.7) [black]{$L$};
						\node at (1.75,2.5) [black]{$A$};
						\node at (3.5,-3.3) [black]{$B$};
						\node at (-3.5,-3.3) [black]{$D$};
						\node at (3.5,3.325) [black]{$a_4$};
						\node at (3,3.8) [black]{$a_3$};
						\node at (-3.5,3.325) [black]{$a_2$};
						\node at (-3.5,4.175) [black]{$a_1$};
						\foreach \i/\j in {d11/a6,d11/a7,d12/a8,d12/a9,a1/b4,a2/b4,a4/b5,a3/b5,a15/a8,a16/a6}{\draw[lefttlearrow={latex}, black]( \i) -- (\j);}
						
						\foreach \i/\j in {d1/u1,d2/u1,d3/u2,d4/u2}{\draw[leftlearrow={latex}, black]( \i) -- (\j);}
						
						\foreach \i/\j/\k/\l in {d5/170/170/u3,d7/180/85/u4,u4/-85/180/d8,d9/180/85/u5,u5/-85/180/d10}{\draw [leftlearrow={latex},black] (\i) to [out=\j,in=\k]  (\l);}
						\draw [leftlearrow={latex},black] (d6) to [out=-170,in=-170]  (-4.99,-0.782);
						\draw[decoration={aspect=-0.05, segment length=1.38mm, amplitude=0.8mm,coil},decorate,line width=1pt] (u3) -- (u6);

						\foreach \i/\j in {d2/a11,d3/a11,d4/a12,d5/a12,d6/a13,d7/a13,d14/a14,d9/b3,a16/a14,a1/b3,d1/a17,d14/a17}{\draw[leftlearrow={latex}, red,dashed]( \i) -- (\j);}

						\draw [middlearrow={latex},red,dashed] (a3) to [out=45,in=100]  (4.1,3) to  [out=-80,in=45] (b2);
						\draw [middlearrow={latex},red,dashed] (a15) to [out=45,in=110]  (4.5,3.2) to  [out=-80,in=30] (b2);
						\foreach \i/\j in {(3.5,0)/b2,(3.5,-1)/b3,(1.85,3.15)/a1,(1.85,3.5)/a2,(1.85,4.35)/a3,(1.85,4)/a4,(1.5,3.75)/a5,(-1.85,3.15)/a6,(-1.85,3.5)/a7,(-1.85,4.35)/a8,(-1.85,4)/a9,(-1.5,3.75)/a10,(0.6,3.15)/a11,(1,3.15)/a12,(1.4,3.15)/a13,(-0.35,3.15)/a14,(1,4.35)/a15,(-1.1,3.55)/a16,(0.2,3.15)/a17,(-3.5,1.7)/d14,(-3.5,1.3)/d1,(-3.5,0.9)/d2,(-3.5,0.5)/d3,(-3.5,0.1)/d4,(-3.5,-0.3)/d5,(-3.5,-0.8)/d6,(-3.5,-1.3)/d7,(-3.5,-1.8)/d8,(-3.5,-2.3)/d9,(-3.5,-2.8)/d10}{\filldraw[black]\i circle (1.5pt)coordinate(\j);}
					
					\end{tikzpicture}\caption*{(b)  $|A|$ is large}
			\end{minipage}}
			\caption{An illustration of how to find the desired path $P$. The white diamonds and black circles indicate the vertices in $\mathcal{B}_A\cup \mathcal{B}_C$ and $\delta$-good vertices of $G$, respectively.  The red dashed edges are obtained by finding common neighbors due to (XY), (CD), (DD) and (AD). For simplicity, the set $C$ is  omitted  and  $P_1,P_2,P^{\ast}$ are placed outside of the partition $(A,B,C,D)$.}
			\label{FIG-Asmalllarge}
		\end{figure}
		
		From the assignment of $\delta$-bad vertices and the fact that $|A|\leqslant n/50$, it follows that $d^-(v,D)> n/50$ for each $v\in \mathcal{B}_A$. Similar with the construction of the path $L_1$,  choosing two in-neighbors of each $v\in \mathcal{B}_A$ in $D$, we may get an antidirected path $Q$ in the remaining $G$ with form 
		$(D\mathcal{B}_ADB)^{|\mathcal{B}_A|}D$ by (DD). Moreover, we have $|Q|\leqslant O(\delta n)$ due to  $|\mathcal{B}_A|\leqslant \delta n$.
		
		Now we extend the path $Q$ to an antidirected path $P_A$ covering all remaining vertices of $A$. Applying (XY) with $(A,B)$ and $(D,A)$ respectively, every two vertices in $A$ has many common out-neighbors in $B$ and many in-neighbors in $D$. So we may easily construct two antidirected paths $P_{AB}=(AB)^{\lceil t/2\rceil}A$ and    $P_{DA}=D(AD)^{\lfloor t/2\rfloor}$, where $t$ is the number of remaining vertices of $A$.  Finally, by (XY) with $(D,A)$, we can connect $P_{DA}$, $P_{AB}$ and $Q$ into $Q^{\ast}$ using common out-neighbors in $B$ of end-vertices of $P_{DA}$ and one end-vertex of $P_{AB}$ and of $Q$, respectively. Noted that  all end-vertices of $Q^{\ast}$ and $P^{\ast}$ are source vertices and, each of $P_1,P_2$ has a sink end-vertex and a source end-vertex. Again, (DD) and (AD) show that we can  get the desired path $P_A$ from  $P_1,P^\ast,Q^\ast,P_2$ by picking common out-neighbors of these source end-vertices in $B$, see Figure \ref{FIG-Asmalllarge} (a) for an illustration. Clearly, $|P_A|=|P^{\ast}|+O(\delta n)$, the end-vertices of $P_A$ belong to $D$ and they are sink vertices of $P_A$. Indeed, the end-vertices  of $P_A$ are the sink end-vertices of $P_1$ and $P_2$. 
		
		\medskip
		\textbf{Case 2: $|A|> n/50$.} 
		\medskip

		In this case, every two vertices of $D$ has at least $|A|-2\delta n$ common out-neighbors in $A$, as stated in the property (XY). This means that the vertices of $A$ are useful when we connect two vertices of $D$. In particular, they will be used to connect the paths $P_1,P_2$ and $P^{\ast}$, as shown in \textbf{Case 1}. In contrast, in this case we connect those paths first and then cover the remaining vertices of $A$.
		
		Similar with the arguments when we construct the path $P_C$, the vertices in $\mathcal{B}_A$ can be divided into three sets $\mathcal{B}_A^1,\mathcal{B}_A^2$ and $\mathcal{B}_A^3$. More precisely, every vertex in $\mathcal{B}_A^1$ has at least $n/50$ in-neighbors in $D$ and,  every vertex in $\mathcal{B}_A^2$ and $\mathcal{B}_A^3$ has at least $n/50$ out- and in-neighbors in $A$, respectively. Noted that by \ref{GA} and the fact that $|A|$ is large, we may move some vertices from $A$ into $\mathcal{B}_A^i$ if necessary so that $|\mathcal{B}_A^i|\geqslant 2$ for each $i\in[3]$. For each $v\in \mathcal{B}_A^1$,  choose two of its in-neighbors in $D$ and similarly, pick two out-neighbors (resp., in-neighbors) of each vertex of $\mathcal{B}_A^2$ (resp., $\mathcal{B}_A^3$) in $A$. Since every two vertices of $D$ has at least $|A|-2\delta n$ common out-neighbors in $A$, there is an antidirected path $L=(D\mathcal{B}_A^{1}DA)^{|\mathcal{B}_A^{1}|}D$ covering all vertices in $\mathcal{B}_A^{1}$.  Let $u_1,u_2$ be the end-vertices of $L$. Clearly, both $u_1$ and $u_2$ are source vertices of $L$.

		Recall that $|\mathcal{B}_A^2|\geqslant 2$ and $|\mathcal{B}_A^3|\geqslant 2$. Assume $a_1,a_2\in \mathcal{B}_A^2$ and let $a_i^{\prime},a_i^{\prime\prime},i\in[2]$ be the chosen out-neighbors of $a_i$ in $A$. Moreover, say $a_3,a_4\in \mathcal{B}_A^3$ and let $a_i^{\prime},a_i^{\prime\prime}$ be the chosen in-neighbors of $a_i$ in $A$. Since every vertex of $A$ has large in-degree in $A$, the vertex $a_i^{\prime}$ has an unused in-neighbor $v_i$ in $A$ for each $i\in[2]$.  It follows by (AD) that $u_2$ and $v_2$ have large common out-degree in $A$.  Again, as  every two vertices of $D$ has at least $|A|-2\delta n$ common out-neighbors in $A$, we may connect the paths $L,P_1$ and $P^{\ast}$ by using vertices in $A$. After this, one may get an antidirected path $L_1^{\prime}$ with form $P_1AP^{\ast}ALAv_2a_2^{\prime}a_2a_2^{\prime\prime}$, see Figure \ref{FIG-Asmalllarge} (b) for an illustration. Applying (XY) with $(A,B)$, we have that $v_1$ and $a_3^{\prime}$ have an unused common out-neighbor in $B$. Furthermore, (AD) implies that the source end-vertex of $P_2$ and $a_4^{\prime}$ have an unused common out-neighbor in $B$. Then $L_2^{\prime}:=a_1^{\prime\prime}a_1a_1^{\prime}v_1Ba_3^{\prime}a_3a_3^{\prime\prime}$ and $L_3^{\prime}:= a_4^{\prime\prime}a_4a_4^{\prime}BP_2$ are two antidirected paths.   Applying (XY) with $(D,A)$ and $(A,B)$, respectively, there are two antidirected paths $L_4^{\prime}:= a_2^{\prime\prime} DA(\mathcal{B}_A^2A)^{|\mathcal{B}_A^2|-2}D(AD)^{\lceil s/2\rceil}a_1^{\prime\prime}$  and $L_5^{\prime}:=a_3^{\prime\prime}BA(\mathcal{B}_A^3A)^{|\mathcal{B}_A^3|-2}(AB)^{\lceil s/2\rceil}a_4^{\prime\prime}$ covering all remaining vertices of $A\cup\mathcal{B}_A^2\cup \mathcal{B}_A^3$ for some $s\in \mathbb{N}$. Then $P_A=L_1^\prime L^\prime_4 L^\prime_2 L^\prime_5 L^\prime_3$ is the desired path.
		
		\medskip
		In both cases, the end-vertices of $P_A$ are sink vertices of $P_A$ and they belong to $D$. Recall that $|P_C|=O(\delta n)$ and the end-vertices of $P_C$ are sink vertices of $D$ also. By (XY) with $(C,D)$, the paths $P_C$ and $P_A$ can be connected by an unused vertex in $C$. Let $P$ be the resulting path, that is, $P=P_CCP_A$. Now we claim that $P$ satisfies the conditions of the lemma.  By the construction of $P_C$ and $P_A$, \ref{L1} and \ref{L2} holds clearly. Moreover, the path $P$ satisfies \ref{L4} as $|P_C|=O(\delta n)$ and all vertices of $V(P_A)\backslash V(P_1\cup P_2\cup \mathcal{B}_A)$ are not in $C$. Recall that $|\mathcal{B}_A\cup \mathcal{B}_C|\leqslant \delta n$, $|P^\ast|=n/200$ and the partition $(A,B,C,D)$ is obtained from the original partition $(A,B,C,D)$ by removing all vertices of $\mathcal{B}_A\cup \mathcal{B}_C$. This means that for any $X\in \{A,B,C,D\}$, the sizes of new $X$ and original $X$ differ by at most $\delta n$. So for the new sets $A,B,C,D$, we have  $|B|+|D|\leqslant  n/2+ O(\delta n)$ and $|A|+|C| \geqslant n/2 -O(\delta n)$ by \ref{NP1} and \ref{NP2}. Therefore, by the fact that $A\cup V(P^\ast)\subseteq V(P)$ and \ref{L4}, we have
		\begin{align*}
			|B\backslash V(P)|+|D\backslash V(P)|& \leqslant n/2 +O(\delta n) - |P^{\ast}| - |A|\\
			&\leqslant |C|+O(\delta n) - n/200\\
			&<|C\backslash V(P)| -n/300,
		\end{align*}
		which proves \ref{L3}.
	\end{proof}
	
	\bigskip
	
	We are now ready to prove Theorem~\ref{THM:non-expander-case}. 
	
	\begin{proof}[{\bf Proof of Theorem \ref{THM:non-expander-case}}]
		
		By Lemma \ref{LEM:non-expander-case}, there exists a constant $\varepsilon$ with $\varepsilon \ll  1$ such that $G$ contains an $\varepsilon$-nice partition $(A,B,C,D)$. Furthermore, by Lemmas  \ref{LEM-goodpath} and \ref{LEM-ACpath}, there exist a constant $\delta$ with $\varepsilon\ll \delta \ll 1$ and an antidirected path $P$ in $G$ satisfying \ref{L1}-\ref{L4}.
		
		Next we  extend $P$ into an antidirected Hamilton cycle. Note that $|C\backslash V(P)| \geqslant n/4 -O(\varepsilon n) - |C\cap V(P)|\geqslant n/5$ by \ref{NP1} and \ref{L4}. Lemma \ref{LEM-randompartition} and \ref{L3} ensure the existence of a partition $C_1,C_2$ of $C\backslash V(P)$ satisfying the following:
		\begin{enumerate}[label =\upshape \textbf{(CP\arabic{enumi})}, ref =\upshape (CP\arabic{enumi})]
			\setlength{\itemindent}{1.5em}
			\item $|C_1| = |B\backslash V(P)| + |D\backslash V(P)| + 2\sqrt{\delta} n$;\label{CP1}
			\item $|C_2|=|C\backslash V(P)|-|C_1|\geqslant n/400$;\label{CP2}
			\item $|d^\pm (v,C_i) - \frac{|C_i|}{|C\backslash V(P)|}d^\pm (v,C\backslash V(P))| \leqslant \varepsilon n$ for each $i\in[2]$ and  $v \in G$.\label{CP3}
		\end{enumerate}

		Let $b \in B \backslash V(P)$ and $d \in D \backslash V(P)$ be two arbitrary vertices. Note that both $b$ and $d$ are $\delta$-good. It follows by \ref{GB} and \ref{GD} that  all but at most $\delta n$ vertices of $C$ are out-neighbors of $b$ and  in-neighbors of $d$.  So $d^+(b,C_1), d^-(d,C_1)\geqslant |C_1|-\delta n$. Hence, for each $b_1,b_2\in B\backslash V(P)$, \ref{CP1} implies that  $|N^+(b_1,C_1)\cap N^+(b_2,C_1)|\geqslant |C_1|-2\delta n>|B\backslash V(P)| + |D\backslash V(P)|+\sqrt{\delta}n$. Similarly, we have $|N^-(d_1,C_1)\cap N^-(d_2,C_1)|>|B\backslash V(P)| + |D\backslash V(P)|+\sqrt{\delta}n$ for each $d_1,d_2\in D\backslash V(P)$. Thus, we can extend $P$ to $P^{\prime}$ by using vertices in $C_1$ to cover all vertices in $(B \cup D)\backslash V(P)$ with form
		$$P^{\prime}=C_1P C_1(DC_1)^{|D\backslash V(P)|}c(BC_1)^{|B\backslash V(P)|} \mbox{ with } c\in C_1.$$ Note that the vertex $c$ exists as each $b\in B$ and each $c_1\in C_1$  have many common out-neighbors in $C_1$. Indeed, we have $|N^+(b,C_1)\cap N^+(c_1,C_1)|\geqslant |C_1|/2-O(\delta n)$ by \ref{GB} and \ref{CP3}.

		Let $x$ and $y$ be the end-vertices of $P^{\prime}$ and let $C_0 = C_2 \cup (C_1 \backslash V(P^{\prime}))\cup\{x,y\}$.  Note that all vertices of $C_0$ are $\delta$-good as all bad vertices are covered by $P^{\prime}$. A simple calculation shows that for each $c\in C_0$, we have
		
		\begin{align*}
			d^\pm(c,C_0)\geqslant d^\pm(c,C_2)&\geqslant  \frac{|C_2|}{|C\backslash V(P)|} d^\pm (c,C\backslash V(P))-\varepsilon n\\
			&\geqslant \frac{|C_2|}{|C\backslash V(P)|} \left(d^\pm(c,C)-|C\cap V(P)|\right)-\varepsilon n\\
			&\geqslant \frac{|C_2|}{|C\backslash V(P)|} \left(|C\backslash V(P)|/2 - O(\delta n)\right)-\varepsilon n\\
			&\geqslant |C_2|/2 - \sqrt{\delta}n \geqslant 4|C_0|/9,
		\end{align*}
		where the above inequalities follow from \ref{GC}, \ref{L4}, \ref{CP2}, and the bound $|C_0\backslash C_2|\leqslant 2\sqrt{\delta}n$.
		
		By Lemma \ref{LEM-semitoexpander} and Theorem \ref{THM-HPfixendv}, there exists an antidirected Hamilton path $P^{\prime\prime}$ in $G[C_0]$ with end-vertices $x,y$. Then $P^{\prime
		}P^{\prime\prime}$ is  the desired antidirected Hamilton cycle of $G$, which completes the proof. 
	\end{proof}

	\section{Proof of Proposition \ref{PROP-degreesharp}}\label{SEC: pf-of-PROP-degreesharp}
	
	In this section, we will give a proof of Proposition \ref{PROP-degreesharp}. First let us recall the statement.
	\medskip
	
	\noindent \textbf{Proposition \ref{PROP-degreesharp} ~}
	For any even integer $n \geqslant4$, there are infinitely many oriented graphs $G$ on $n$ vertices  with $\sigma_{+-}(G)= \lceil (3n+2)/4 \rceil - 1$ which does not contain  antidirected Hamilton cycles.
	\begin{proof}
		Let $G$ be any oriented graph shown in Figure \ref{FIG-degreesharp}. As shown in Table \ref{TAB-degreesharp}, we habe $\sigma_{+-}(G)= \lceil (3n+2)/4 \rceil - 1$. Next we claim that $G$ has no antidirected Hamilton cycles. Suppose to the contrary that $G$ has an antidirected Hamilton cycle $L=v_1v_2\cdots v_nv_1$. 
		
		First we consider the case that $G$ is isomorphic to the graph shown in Figure \ref{FIG-degreesharp} (a) or (b). Assume without loss of generality that $v_1\in A$. Suppose first that $v_1$ is a source vertex of $L$, that is, $v_1v_2,v_1v_n\in E(G)$. Since $v_1$ has at most one out-neighbor in $V(G)\backslash (A\cup B)$, one of $v_2$ and $v_n$, say $v_2$, must be in $A\cup B$. Then $v_2$ is also a sink vertex of $L$, i.e. $v_3v_2\in E(G)$. Thus $v_3$ must be in $A\cup D$. Continue this procedure until all vertices are considered, we get that all vertices of $L$ with odd indices must belong to $A\cup D$ and vertices with even indices are in $A\cup B$. This implies that the antidirected Hamilton cycle $L=v_1v_2\cdots v_nv_1$ only uses the vertices in $A\cup B\cup D$, a contradiction. For the case that  $v_1$ is a sink vertex of $L$,  one may  verify similarly  that $L$  only uses the vertices in $A\cup B\cup D$, a contradiction again.

		So it suffices to consider the case that $G$ is isomorphic to the digraph shown in Figure \ref{FIG-degreesharp} (c). In this case, observe that all edges between $B$ and $D$ are useless when we embedded the antidirected cycle $L$. Indeed, suppose  $v_1v_2\in E(L)$ and it is embedded onto an edge from $D$ to $B$, then $v_3$ must be in $D$ as $v_2$ is a sink vertex of $L$. Similarly,  we get that all vertices of $L$ with odd indices must belong to $D$ and vertices with even indices are in $B$, a contradiction. Therefore, the cycle $L$ must be embedded with form $(CD)^sC\cdots C(BC)^tC\cdots C$, where the second ``$C\cdots C$'' may be omitted. This  means that the size of $C$ should be at least $|B|+|D|+2$. However, the digraph in (c) has order $4s+2$ with $|C|=2s, |B|=|D|=s+1$, a contradiction. Hence $G$ has no antidirected Hamilton cycles and this completes the proof.
	\end{proof}

	\begin{table}[H]
		\centering
		\caption{ \small The orders of $A,B,C,D$ and the Ore-type condition of $G$.}
		\label{TAB-degreesharp}
		\begin{tabular}{m{1.5cm}<{\centering}m{1.5cm}<{\centering}m{1.5cm}<{\centering}m{2cm}<{\centering}m{1.5cm}<{\centering}}
			\hline
			$n$  &  $\sigma_{+-}(G)$ & $|A|$ & $|B|=|D|$ & $|C|$ \\
			\hline
			$8s+6$ & $6s+4$ & $2s+1$ & $2s+2$ & $2s+1$ \\
			$4s$ & $3s$ & $1$ & $s$ & $2s-1$\\
			$4s+2$ & $3s+1$ & $0$ & $s+1$ & $2s$\\
			\hline
		\end{tabular}
	\end{table}

	\begin{figure}[H]
		\subfigure{
			\begin{minipage}[t]{0.32\linewidth}
				\centering
				\begin{tikzpicture}[black,line width=0.8pt,scale=0.8]
					\draw (0,1.6) circle (0.4);
					\draw (0,1.6) ellipse (1 and 0.6);
					\coordinate [label=center:$A$] (A) at (0,1.6);
					\draw (-2.3,0) ellipse (0.8 and 0.6);
					\coordinate [label=center:$D$] (D) at (-2.3,0);
					\draw (0,-1.6) circle (0.4);
					\draw (0,-1.6) ellipse (1 and 0.6);
					\coordinate [label=center:$C$] (C) at (0,-1.6);
					\draw (2.3,0) ellipse (0.8 and 0.6);
					\coordinate [label=center:$B$] (B) at (2.3,0);
					\draw [-stealth, line width=1.2pt] (0.05,1.2) -- (-0.05,1.2);
					\draw [-stealth, line width=1.2pt] (0.05,-1.2) -- (-0.05,-1.2);
					\draw[-stealth] [line width=2.5pt](1,1.3) -- (2,0.65);
					\draw[-stealth] [line width=2.5pt](-1,-1.3) -- (-2,-0.65);
					\draw[-stealth] [line width=2.5pt](2,-0.65) -- (1,-1.3);
					\draw[-stealth] [line width=2.5pt](-2,0.65) -- (-1,1.3);
					\filldraw[white](-1.45,0.1) circle (1pt)node[](u){};
					\path[draw, -stealth, line width=1.2pt] (u) edge[bend left=15] (1.45,0.1);      
					\filldraw[white](1.45,-0.1) circle (1pt)node[](v){};
					\path[draw, -stealth, line width=1.2pt] (v) edge[bend left=15] (-1.45,-0.1);
				\end{tikzpicture}\caption*{(a)  }
		\end{minipage}}
		\subfigure{
			\begin{minipage}[t]{0.32\linewidth}
				\centering
				\begin{tikzpicture}[black,line width=0.8pt,scale=0.8]
					\draw (0,-1.6) circle (0.4);
					\draw (0,-1.6) ellipse (1 and 0.6);
					\coordinate [label=center:$C$] (C) at (0,-1.6);
					\draw (-2.3,0) ellipse (0.8 and 0.6);
					\coordinate [label=center:$D$] (D) at (-2.3,0);
					\coordinate [label=center:$A$] (A) at (0,2);
					\draw (2.3,0) ellipse (0.8 and 0.6);
					\coordinate [label=center:$B$] (B) at (2.3,0);
					\draw [-stealth, line width=1.2pt] (0.05,-1.2) -- (-0.05,-1.2);
					
					\draw[-stealth] [line width=2.5pt](2,-0.65)--(1,-1.3) ;
					\draw[-stealth] [line width=2.5pt] (-2,0.65)--(-0.3,1.3);
					\draw[-stealth] [line width=2.5pt](0.3,1.3)--(2,0.65) ;
					\filldraw[black](0,1.3) circle (2pt)node[](){};
					
					\draw[-stealth] [line width=2.5pt](-1.4,0) -- (1.4,0);
					
					\draw[-stealth] [line width=1.2pt](-0.1,1.1) -- (-0.6,-1.6);
					\draw[-stealth] [line width=1.2pt](0.6,-1.6) -- (0.1,1.1);

					\draw[-stealth] [line width=2.5pt] (-1,-1.3) --(-2,-0.65);
					\filldraw[white](-1.45,0.1) circle (1pt)node[](u){};     
					\filldraw[white](1.45,-0.1) circle (1pt)node[](v){};
					
				\end{tikzpicture}\caption*{(b)  }
		\end{minipage}}
		\subfigure{
			\begin{minipage}[t]{0.32\linewidth}
				\centering
				\begin{tikzpicture}[black,line width=0.8pt,scale=0.8]
					\draw (0,-3) circle (0.4);
					\draw (0,-3) ellipse (1.1 and 0.7);
					\coordinate [label=center:$C$] (C) at (0,-3);
					\draw (-2.3,0) ellipse (0.8 and 0.6);
					\coordinate [label=center:$D$] (D) at (-2.3,0);
					
					\draw (2.3,0) ellipse (0.8 and 0.6);
					\coordinate [label=center:$B$] (B) at (2.3,0);
					\draw [-stealth, line width=1.2pt] (0.05,-2.6) -- (-0.05,-2.6);
					\draw[-stealth] [line width=2.5pt] (2,-0.65)--(0.5,-2.3);
					\draw[-stealth] [line width=2.5pt](-0.5,-2.3)--(-2,-0.65);
					\draw[-stealth] [line width=2.5pt] (-1.4,0)--(1.4,0);
					\filldraw[white](-1.45,0.1) circle (1pt)node[](u){};
					\filldraw[white](1.45,-0.1) circle (1pt)node[](v){};
				\end{tikzpicture}\caption*{(c)  }
		\end{minipage}}
		\caption{\small The oriented graphs in Proposition \ref{PROP-degreesharp}   with order $8s+6,4s$ and $4s+2$, respectively. The size of each sets is given in Table \ref{TAB-degreesharp}. The bold edges indicate that all possible edges are present and have the directed shown.  Each of $A$ and $C$ spans an almost regular tournament, that is, the in-degree and out-degree of every vertex differ by at most one. Both $B$ and $D$ are empty sets and, in (a) the oriented graph induced by $B$ and $D$ is an almost regular bipartite tournament. In (b), $A$ has order one and the vertex in $A$ has exactly one in-neighbor and out-neighbor in $C$.} 
		\label{FIG-degreesharp}
	\end{figure}

	\section{Conclusion}\label{SEC-remark}
	
	In this paper, we have established an Ore-type degree condition for the existence of antidirected Hamilton cycles in oriented graphs. More precisely, we prove that for sufficiently large even integer $n$, every oriented graph $G$ on $n$ vertices with $\sigma_{+-}(G)\geqslant (3n+2)/4$ contains an antidirected Hamilton cycle. Furthermore, we construct three surprising counterexamples showing that the degree condition is best possible. Our result contributes to the broader program of extending classical Hamiltonian results in graphs to oriented graphs, particularly for non-standard cycle orientations. It also refines earlier work on degree conditions for Hamiltonicity in digraphs and oriented graphs, and enhances our understanding of how local degree conditions can enforce global structural properties.

	An oriented graph $G$ is said to be \emph{vertex-pancyclic} if for every vertex $v\in V(G)$ and every integer $3\leqslant l\leqslant n$, there is a directed cycle of length $l$ containing $v$. Bondy \cite{bondy1971} proposed the meta-conjecture: almost any nontrivial condition on graphs implying a graph is Hamiltonian also implies that it is pancyclic (with possibly a few exceptional families of graphs). This meta-conjecture has been verified for several sufficient conditions, this motivates us to examine these sufficient conditions for vertex-pancyclicity, since vertex-pancyclicity implies pancyclicity. Recall that Chang et al. proved that there exists an integer $n_0$ such that every oriented graph $G$ on $n\geqslant n_0$ vertices with $\sigma_{+-}(G)\geqslant(3n-3)/4$ contains a directed Hamilton cycle (Theorem \ref{THM-changarXiv2025}).  We conjecture that the same condition also implies vertex-pancyclicity. 
	
	\begin{conjecture}
		There exists an integer $n_0$ such that every oriented graph $G$ on $n\geqslant n_0$ vertices with $\sigma_{+-}(G)\geqslant(3n-3)/4$ is vertex-pancyclic.
	\end{conjecture}

	One direction is to investigate whether the bound $(3n-3)/4$ can be improved under additional assumptions or for specific classes of oriented graphs. The following result provides some support for this problem, as well as to the above conjecture.
	
	\begin{theorem}[\cite{bangJGTta}]\label{bangJGT}
		Let $G$ be an oriented graph on $n \geqslant 9$ vertices. If  $G$ is of minimum degree $n-2$ and  $\sigma_{+-}(G)\geqslant n-3$, then $G$ is vertex-pancyclic.
	\end{theorem}
	
	The degree condition in Theorem \ref{bangJGT} is best possible in the sense
	that the conclusion fails if any one of the three conditions $n\geqslant 9$, $\delta(G)\geqslant n-2$ or $\sigma_{+-}(G)\geqslant n-3$ is relaxed.  For further details, see \cite{bangJGTta} and \cite{songJGT18}.

	Another natural and challenging problem is to determine whether a similar Ore-type degree condition can guarantees the existence of cycles with all possible orientations in oriented graphs. In this direction, we propose the following conjecture:
	
	\begin{conjecture}
		There exists an integer $n_0$ such that every oriented graph $G$ on $n\geqslant n_0$ vertices with $\sigma_{+-}(G)\geqslant(3n+2)/4$  contains all possible orientations of a (undirected) Hamilton cycle.
	\end{conjecture}

	\bibliographystyle{plain}

\begin{thebibliography}{99}
	
	\bibitem{bangJGTta}
	J.~Bang-Jensen and Y.~Guo.
	\newblock {A note on vertex pancyclic oriented graphs}.
	\newblock {\em J. Graph Theory}, 31:313--318, 1999.
	
	\bibitem{bang2009}
	J.~Bang-Jensen and G.~Gutin.
	\newblock {\em {Digraphs: Theory, Algorithms and Applications}}.
	\newblock Springer-Verlag, London, 2nd edition, 2009.
	
	\bibitem{bondy1971}
	J.~A. Bondy.
	\newblock Pancyclic graphs, in: Proceedings of the second louisiana conference
	on combinatorics, graph theory and computing.
	\newblock {\em Louisiana State University, Baton Rouge}, pages 167--172, 1971.
	
	\bibitem{changarXiv2025}
	Y.~Chang, Y.~Cheng, T.~Dai, Q.~Ouyang, and G.~Wang.
	\newblock {An exact Ore-degree condition for Hamilton cycles in oriented
		graphs}.
	\newblock {\em arXiv preprint arXiv:2507.04273}, 2025.
	
	\bibitem{chibaGC34}
	S.~Chiba and T.~Yamashita.
	\newblock Degree conditions for the existence of vertex-disjoint cycles and
	paths: {A} survey.
	\newblock {\em Graphs Combin.}, 34(1):1--83, 2018.
	
	\bibitem{chvatalJCTB12}
	V.~Chvátal.
	\newblock On {H}amilton's ideals.
	\newblock {\em J. Combin. Theory Ser. B}, 12:163--168, 1972.
	
	\bibitem{debiasioSJDM29}
	L.~DeBiasio, D.~K\"uhn, T.~Molla, D.~Osthus, and A.~Taylor.
	\newblock Arbitrary orientations of {H}amilton cycles in digraphs.
	\newblock {\em SIAM J. Discrete Math.}, 29(3):1553--1584, 2015.
	
	\bibitem{debiasioEJC22}
	L.~DeBiasio and T.~Molla.
	\newblock Semi-degree threshold for anti-directed {H}amiltonian cycles.
	\newblock {\em Elec. J. Combin.}, 22(4), 2015.
	
	\bibitem{debiasioarXiv2025}
	L.~DeBiasio and A.~Treglown.
	\newblock {Arbitrary orientations of Hamilton cycles in directed graphs of
		large minimum degree}.
	\newblock {\em arXiv preprint arXiv.2505.09793}, 2025.
	
	\bibitem{diracPLMS2}
	G.~A. Dirac.
	\newblock {Some theorems on abstract graphs}.
	\newblock {\em Proc. London Math. Soc.}, 2(3):69--81, 1952.
	
	\bibitem{diwanDM311}
	A.~Diwan, J.~B. Frye, M.~J. Plantholt, and S.~K. Tipnis.
	\newblock A sufficient condition for the existence of an anti-directed 2-factor
	in a directed graph.
	\newblock {\em Discrete Math.}, 311(21):2556--2562, 2011.
	
	\bibitem{fan1984new}
	G.~Fan.
	\newblock New sufficient conditions for cycles in graphs.
	\newblock {\em J. Combin. Theory, Ser. B}, 37(3):221--227, 1984.
	
	\bibitem{ghouilaCRASP25}
	A.~Ghouila-Houri.
	\newblock {Une condition suffisante d'existence d'un circuit hamiltonien}.
	\newblock {\em C.R. Acad. Sci. Paris}, 251:495--497, 1960.
	
	\bibitem{gouldGC19}
	R.~Gould.
	\newblock Advances on the {H}amiltonian problem: {A} survey.
	\newblock {\em Graphs Combin.}, 19:7--52, 2003.
	
	\bibitem{haggkvist1995oriented}
	R.~H{\"a}ggkvist and A.~Thomason.
	\newblock {Oriented Hamilton cycles in digraphs}.
	\newblock {\em J. Graph Theory}, 19(4):471--479, 1995.
	
	\bibitem{keevashJLMS79}
	P.~Keevash, D.~K\"uhn, and D.~Osthus.
	\newblock An exact minimum degree condition for {H}amilton cycles in oriented
	graphs.
	\newblock {\em J. London Math. Soc.}, 79(1):144--166, 2009.
	
	\bibitem{kellyCPC17}
	L.~Kelly, D.~K\"uhn, and D.~Osthus.
	\newblock A {D}irac-type result on {H}amilton cycles in oriented graphs.
	\newblock {\em Combin. Prob. Comput.}, 17(5):689--709, 2008.
	
	\bibitem{kuhnAM237}
	D.~K\"uhn and D.~Osthus.
	\newblock {Hamilton decompositions of regular expanders: {A} proof of {K}elly's
		conjecture for large tournaments}.
	\newblock {\em Adv. Math.}, 237:62--146, 2013.
	
	\bibitem{oreAMM67}
	O.~Ore.
	\newblock Note on {H}amilton circuits.
	\newblock {\em Am. Math. Mon.}, 67(1):55, 1960.
	
	\bibitem{posaPMIHAS7}
	L.~P{\'o}sa.
	\newblock {A theorem concerning Hamiltonian lines}.
	\newblock {\em Publ. Math. Inst. Hung. Acad. Sci.}, 7:225--226, 1962.
	
	\bibitem{songJGT18}
	Z.~Song.
	\newblock {Pancyclic oriented graphs}.
	\newblock {\em J. Graph Theory}, 18(5):461--468, 1994.
	
	\bibitem{stein2024antidirected}
	M.~Stein and C.~Z{\'a}rate-Guer{\'e}n.
	\newblock Antidirected subgraphs of oriented graphs.
	\newblock {\em Combin. Prob. Comput.}, 33(4):446--466, 2024.
	
	\bibitem{taylor2013}
	A.~Taylor.
	\newblock The regularity method for graphs and digraphs.
	\newblock MSci thesis, School of Mathematics, University of Birmingham, UK,
	2013; also available from arXiv:1406.6531.
	
	\bibitem{wangarxiv2025}
	G.~Wang, Y.~Wang, and Z.~Zhang.
	\newblock Arbitrary orientations of cycles in oriented graphs.
	\newblock {\em arXiv preprint arXiv:2504.09794}, 2025.
	
	\bibitem{woodallPLMS24}
	D.~Woodall.
	\newblock {Sufficient conditions for cycles in digraphs}.
	\newblock {\em Proc. London Math. Soc.}, 24:739--755, 1972.
	
\end{thebibliography}

\end{document}